\documentclass[12pt]{amsart}

\usepackage{amssymb}
\usepackage{amsmath}
\usepackage{amsthm}
\usepackage{enumitem}
\usepackage{graphicx}
\usepackage{tikz}
\usepackage[all]{xy}

\definecolor{untgreen}{RGB}{5,144,51} 
\usepackage[breaklinks=true, %ocgcolors,
colorlinks, allcolors=untgreen, linktocpage ]{hyperref}

\numberwithin{equation}{subsection} 
\usepackage[capitalize]{cleveref}
\crefname{equation}{}{} % redefine cleveref eqn label
\crefname{subsection}{}{} % redefine cleveref subsection label
\crefname{section}{\S\kern -.5 ex}{\S\S\kern -1 ex}

\usepackage{verbatim}
\usepackage{color}
\usepackage{multirow}

\theoremstyle{plain}
\newtheorem{claim*}[subsection]{Claim}

\newtheorem{corollary}[subsection]{Corollary} 
\newtheorem{lemma}[subsection]{Lemma}
\newtheorem{prop}[subsection]{Proposition}
\newtheorem{theorem}[subsection]{Theorem} 
\newtheorem*{theorem*}{Theorem}

\theoremstyle{definition}%

\theoremstyle{remark}
\newcommand{\bc}{\begin{center}}
\newcommand{\ec}{\end{center}}
\newcommand{\CC}{\ensuremath{\mathbb C}}

\newcommand\Ind{\operatorname{Ind}}
\newcommand \abs[1] {\left | #1 \right |} 
\newcommand \ccl{\operatorname{ccl}}

\newcommand \la{\langle} %left angle bracket
\newcommand \ra{\rangle} %right angle bracket

\newcommand\spn{\operatorname{span}}

\newcommand\D{\operatorname{D}}
\newcommand\Stab{\operatorname{Stab}}

\newcommand \xlp {x_{|\lambda|}}

\newcommand \elt {\widetilde e_\lambda}

\newcommand\id{\operatorname{id}}
\newcommand{\ds}{\operatorname{ds}}
\newcommand\inverse{^{-1}}

\newcommand\CalC{\mathcal C}
\newcommand\CalP{\mathcal P}
\newcommand\CalS{\mathcal S}
\newcommand\cfC{\operatorname{cf}_{\mathbb C}}

\allowdisplaybreaks %% An awesome macro no tex file should be without

\begin{document}

%%%%%%%%%%%%%%%%%%%%%%%%%%%%%%%%%%%%%%%%%%%%%%%%%%%%%%%%%%%%%%%%%%%%%%
%%%%%%%%%%%%% top matter stuff
%%%%%%%%%%%%%%%%%%%%%%%%%%%%%%%%%%%%%%%%%%%%%%%%%%%%%%%%%%%%%%%%%%%%%%
\title[A decomposition of the group algebra]{A decomposition of the group
  algebra of a hyperoctahedral group}

%%%%% author two information
\author{J.~Matthew Douglass}
\address{Department of Mathematics\\
University of North Texas\\ 
Denton, TX 76203, USA}
%\curraddr{}
\thanks{This work was partially supported by a grant from the Simons
  Foundation (Grant \#245399 to J.M.~Douglass). J.M.~Douglass would like to
  acknowledge that some of this material is based upon work supported by
  (while serving at) the National Science Foundation.}
\email{douglass@unt.edu}

%%%%% author one information
\author{Drew E. Tomlin}
\address{Department of Mathematics\\
  University of North Texas\\
  Denton, TX 76203, USA}
%\curraddr{}
%\thanks{}
\email{drewtillis@my.unt.edu}


\begin{abstract}
  The descent algebra of a finite Coxeter group $W$ is a subalgebra of the
  group algebra defined by Solomon. Descent algebras of symmetric groups
  have properties that are not shared by other Coxeter groups. For instance,
  the natural map from the descent algebra of a symmetric group to its
  character ring is a surjection with kernel equal the Jacobson
  radical. Thus, the descent algebra implicitly encodes information about
  the representations of the symmetric group, and a complete set of
  primitive idempotents in the character ring leads to a decomposition of
  the group algebra into a sum of right ideals indexed by
  partitions. Stanley asked whether this decomposition of the regular
  representation of a symmetric group could be realized as a sum of
  representations induced from linear characters of centralizers. This
  question was answered positively by Bergeron, Bergeron, and Garsia, using
  a connection with the free Lie algebra on $n$ letters, and independently by
  Douglass, Pfeiffer, and R\"ohrle, who connected the decomposition with the
  configuration space of $n$-tuples of distinct complex numbers.

  The Mantaci-Reutenauer algebra of a hyperoctahedral group is a subalgebra
  of the group algebra that contains the descent algebra.  Bonnaf\'e and
  Hohlweg showed that the natural map from the Mantaci-Reutenauer algebra to
  the character ring is a surjection with kernel equal the Jacobson
  radical. In 2008, Bonnaf\'e asked whether the analog to Stanley's question
  about the decomposition of the group algebra into a sum of induced linear
  characters holds. In this paper, we give a positive answer to Bonnaf\'e's
  question by explicitly constructing the required linear characters.
\end{abstract}

\maketitle

%%%%%%%%%%%%%%%%%%%%%%%%%%%%%%%%%%%%%%%%%%%%%%%%%%%%%%%%%%%%%%%%%%%%%%
%%%%%%%%%%%%% article body...
%%%%%%%%%%%%%%%%%%%%%%%%%%%%%%%%%%%%%%%%%%%%%%%%%%%%%%%%%%%%%%%%%%%%%%

%%%%%%%%%%%%%%%%%%%%%%%%%%%%%%%%%%%%%%%%%%%%%%%%%%%%%%%%%%%%%%%
%%%% Section 1 Introduction %%%%%%%%%%%
%%%%%%%%%%%%%%%%%%%%%%%%%%%%%%%%%%%%%%%%%%%%%%%%%%%%%%%%%%%%%%%
\section{Introduction}\label{sec:intro}

Let $W$ be a Coxeter group with a generating set of simple reflections
$S$. For a subset $I$ of $S$, $W_I = \langle I \rangle$ is the standard
parabolic subgroup of $W$ generated by $I$. Then $W_I$ has a set of minimal
length coset representatives $X_I$, where length is the Coxeter length. Set
$x_I = \sum_{w\in X_I} w$. Then $\{\, x_I \mid I\subseteq S \,\}$ forms a
basis for a subalgebra of the group algebra $\mathbb{C} W$ called the
\emph{descent algebra} of $W$ (see \cite{solomon:mackey}).

When $W=S_n$ is a symmetric group, the descent algebra has two particularly
nice properties. First, the natural map from the descent algebra of $S_n$ to
the character ring of $S_n$ is a surjection with kernel equal to the
Jacobson radical, and second, a complete set of orthogonal idempotents of
the descent algebra determines a decomposition of the group algebra of $S_n$
as a direct sum of induced representations of linear characters of
centralizers. This decomposition of the group algebra has applications in
surprisingly different contexts; see
\cite{bergeronbergerongarsia:idempotents} for connections with the free Lie
algebra, \cite{hanlon:action} for an application to Hochschild homology, and
\cite{douglasspfeifferroehrle:cohomology} and
\cite{churchfarb:representation} for connections with the cohomology of the
configuration space of $n$ distinct points in the plane. The first property,
that the natural map from the descent algebra to the character ring is a
surjection, only holds for symmetric groups. It is an open problem (the
Lehrer-Solomon conjecture, see
\cite{bishopdouglasspfeifferroehrle:computationsIII}) to determine whether
the second property, that a complete set of orthogonal idempotents of the
descent algebra determines a decomposition of the group algebra as a direct
sum of induced representations of linear characters of centralizers, holds
for a general Coxeter group $W$.

In this paper we consider hyperoctahedral groups and extensions of their
descent algebras known as \emph{Mantaci-Reutenauer algebras.} The
Mantaci-Reutenauer algebra of a hyperoctahedral group $W_n$ has the property
that the natural map to the character ring of $W_n$ is a surjection with
kernel equal to the Jacobson radical \cite{bonnafehohlweg:generalized}. The
main result in this paper is that a complete set of orthogonal idempotents
of the Mantaci-Reutenauer algebra of $W_n$ determines a decomposition of the
group algebra of $W_n$ as a direct sum of induced representations of linear
characters of centralizers. This answers a question raised by Bonnaf\'e
\cite[\S10]{bonnafe:mantaci}. It would be interesting to find a connection
between the results in this paper and the corresponding conjectural result
\cite{bishopdouglasspfeifferroehrle:computationsIII} for the descent algebra
of $W_n$.

The idempotents in the Mantaci-Reutenauer algebra that play a central role
in this paper were constructed by Vazirani \cite{vazirani:thesis},
generalizing a construction given by Garsia and Reutenauer
\cite{garsiareutenauer:decomposition} for symmetric
groups. Mantaci-Reutenauer algebras are defined for more general wreath
products, and it seems likely that Vazirani's construction, as well as the
results in this paper, can be extended to the complex reflection groups
$G(r,1,n)$ for $r>2$.

The rest of this paper is organized as follows. In the next section we set
out some notation and state the main theorems. Proofs of the main results
are then completed in subsequent sections.

%%%%%%%%%%%%%%%%%%%%%%%%%%%%%%%%%%%%%%%%%%%%%%%%%%%%%%%%%%%%%%%
%%%% Section 2 Mantaci-Reutenauer Algebra %%%%%%%%%%%
%%%%%%%%%%%%%%%%%%%%%%%%%%%%%%%%%%%%%%%%%%%%%%%%%%%%%%%%%%%%%%%

\section{A decomposition of \texorpdfstring{$\CC
    W_n$}{CCWn}} \label{sec:defandnot}

To begin, we need some notation. First, throughout this paper, $n$ is a
positive integer; for a positive integer $k$, set $[k]=\{1, \dots, k\}$; for
convenience the additive inverse of an integer will frequently be denoted by
an overbar (so for example $\overline 3 = -3$ and $\overline{-3}= 3$); and
for $J\subseteq \mathbb Z$, the notation $J=\{j_1< j_2< \dotsm < j_k\}$
indicates that $J=\{j_1, j_2, \dots, j_k\}$ and $j_1< j_2< \dotsm <j_k$.

%%%%%%%%%%%%%%%%%%%%%%%%%%%%%%%%%%%%%%%%%%%%%%%%%%%%%%%%%%%%%%%%%%%%%
\subsection{Compositions and partitions} \label{subsec:compandpart} 

A \emph{composition of $n$} is a tuple of positive integers $p=( p_1,
\ldots, p_k)$ such that $\sum_i p_i = n$. A \emph{partition of $n$} is a
composition in which the entries of the tuple are nonincreasing. More
generally, a \emph{signed composition of $n$} is a tuple of nonzero integers
$p=( p_1, \ldots, p_k )$ such that $\sum_i \abs{p_i} = n$. A \emph{signed
  partition of $n$} is a signed composition in which the positive entries of
the tuple appear first, in nonincreasing order, followed by the negative
entries, in order of nonincreasing absolute value. The entries of a signed
composition $p$ are called \emph{parts} of $p$. Define
\begin{itemize}
\item $\CalC(n)$ to be the set of compositions of $n$,
\item $\CalP(n)$ to be the set of partitions of $n$,
\item $\CalS\CalC(n)$ to be the set of signed compositions of $n$, and
\item $\CalS\CalP(n)$ to be the set of signed partitions of $n$.
\end{itemize}

For $p=(p_1, p_2, \ldots, p_k)\in \CalS\CalC(n)$ set $\widehat p_0 = 0$, and
for $i\in [k]$ define
\[
\widehat p_i = \sum_{j=1}^i \abs{p_j} \quad \text{and} \quad P_i =
\{\,\widehat p_{i-1} +l\mid l=1, \dots, \abs{p_i} \,\} = \{\widehat p_{i-1}
+1, \ldots, \widehat p_i\}.
\]
The subsets $P_i$ of $[n]$ will be referred to as ``blocks'' of $p$.

Signed partitions will frequently be written as
\[
\lambda= \left( \lambda_1, \ldots, \lambda_a, \lambda_{a+1}, \ldots,
  \lambda_{a+b} \right)
\]
to indicate that $\lambda_1, \ldots, \lambda_a $ are the positive parts of
$\lambda$ and $\lambda_{a+1}, \ldots, \lambda_{a+b}$ are the negative
parts. With the conventions above,
\[
\widehat \lambda_i = \sum_{j=1}^i \abs{\lambda_j} \quad\text{and}\quad
\Lambda_i = \{\widehat \lambda_{i-1} +1, \widehat \lambda_{i-1} +2, \ldots,
\widehat \lambda_i \}.
\]

For $p\in \CalS\CalC(n)$ let $\overleftarrow p$ be the signed partition of
$n$ formed by rearranging the parts of $p$.

Suppose $k$ is a positive integer and consider the set of signed
compositions of $n$ with $k$ parts. The symmetric group $S_{k}$ acts on this
set by permuting the parts of a signed composition, and the set of signed
partitions of $n$ with $k$ parts forms a set of orbit representatives for
this action.  If $p$ has $k$ parts, let $\Stab(p)$ be the stabilizer of $p$
in $S_{k}$.

For example, $p=( \overline{1}, 3, \overline{2}, 1, 3, \overline{1})$ is a
signed composition of eleven with six parts, $\overleftarrow p = \left( 3,
  3, 1, \overline{2}, \overline{1}, \overline{1} \right)$, and $\Stab(p)$ is
isomorphic to the Klein four group.

If a signed composition $p$ is fixed, then $\xi_i$ will denote the sign of
$p_i$, where the sign of a positive number is $+$ and the sign of a negative
number is $-$.

%%%%%%%%%%%%%%%%%%%%%%%%%%%%%%%%%%%%%%%%%%%%%%%%%%%%%%%%%%%%%%%%%%%%%%%%%%%

\subsection{Hyperoctahedral groups} \label{subsec:hypgp}

A \emph{signed permutation of $n$} is a permutation $w$ of the set $\{1, 2,
\ldots, n\} \amalg \{\overline{1}, \ldots, \overline{n}\}$ such that $w
(\overline{a} )= \overline{ w(a)}$ for $a$ in $[n]$. Signed permutations of
$n$ naturally form a group under composition, called the
\emph{$n^{\text{th}}$ hyperoctahedral group} and denoted by $W_n$.

We identify $S_n$ with the subgroup of $W_n$ consisting of all signed
permutations $w$ such that $w([n])=[n]$. For a subset $P$ of $[n]$, let
$S_P$ and $W_P$ denote the subgroups of $S_n$ and $W_n$, respectively, that
fix $[n]\setminus P$ pointwise. Then $S_P= S_n \cap W_P \subseteq
W_n$. Similarly, for an integer $m\leq n$, we identify $S_m$ with the
subgroup $S_{[m]}$ of $S_n$, and $W_{m}$ with the subgroup $W_{[m]}$.

In this paper, $W$, $S$, and several other symbols can have four types of
subscripts: a positive integer, usually $m$, $n$, or $|\lambda_i|$; a
(signed) composition, usually $p$, $|p|$, or $|\lambda|$; a (signed)
partition, usually $\lambda$; or a subset of $[n]$, usually $P_i$ or
$\Lambda_i$. The meaning should always be clear from context.

Signed permutations may be represented in two row (or function) notation,
one row notation, or cycle notation. The conventions we will use in this
paper are most clearly demonstrated with an example. In two row, one row,
and cycle notation, respectively,
\[
w= \left( \begin{array}{cccccc}
            1 & 2 & 3 & 4 & 5 & 6  \\
            2 & 3 & \overline{1} & 4 & \overline{6} & \overline{5}
          \end{array} \right) = 2\, 3\,  \overline{1}\, 4\, \overline{6}\,
        \overline{5} = \left( 1 \; \; 2 \; \; 3 \right)^- (4) ( 5\:\:
        \overline{6})
\]
is the signed permutation that maps $1$ to $2$, $2$ to $3$, $3$ to $-1$, and
so on. Here the superscript ${}^-$ in cycle notation denotes a negative
cycle.  Given $a_1, a_2, \dots, a_r$, $\left( a_1 \; \; a_2 \; \; \cdots \;
  \; a_r \right)^-$ is called a \emph{negative $r$-cycle} and denotes the
signed permutation that maps $a_j$ to $a_{j+1}$ for $j\in [r-1]$ and maps
$a_r$ to $\overline{a_1} =-a_1$. Note that a negative $r$-cycle has order
$2r$ as an element of $W_n$.

Each signed permutation $w$ has a \emph{signed cycle type} that is the
signed partition $\lambda$ for which the positive parts of $\lambda$ are the
lengths of the positive cycles in the cycle decomposition of $w$, and the
negative parts of $\lambda$ are the lengths of the negative cycles in the
cycle decomposition of $w$. For example, the signed cycle type of $w= \left(
  1 \; \; 2 \; \; 3 \right)^- (4) \left( 5 \; \; \overline{6} \right)$ is
the signed partition $( 2, 1, \overline{3})$. Note that two signed
permutations in $W_n$ are conjugate if and only if they have the same signed
cycle type.

For a positive integer $i$, let $s_i$ be the positive two-cycle $(i \;\;
i+1)$ that switches $i$ and $i+1$, and let $t_i$ be the negative one-cycle
$(\, i\,)^-$ that sends $i$ to $\overline{\imath}$. Then $\{t_1, s_1, s_2,
\ldots, s_{n-1}\}$ is a set of Coxeter generators of $W_n$.

Finally, let $w_{0,n}$ be the ``longest element'' in $W_n$, so 
\[
w_{0,n}= t_1 \dotsm t_n\quad\text{and} \quad w_{0,n}(a) = \overline{a}
\]
for $a\in [n]$. Note that $\langle w_{0,n}\rangle$ is the center of
$W_n$. Similarly, for $P\subseteq [n]$ define $w_{0,P}= \prod_{j\in P}
t_{j}$ in $W_P$.

%%%%%%%%%%%%%%%%%%%%%%%%%%%%%%%%%%%%%%%%%%%%%%%%%%%%%%%%%%%%%%%%%%%%%%%%
\subsection{Mantaci-Reutenauer algebras}

Mantaci-Reutenauer algebras for hyperoctahedral groups were first defined by
Mantaci and Reutenauer
\cite{mantacireutenauer:generalization}. Subsequently, Bonnaf\'e and Hohlweg
\cite{bonnafehohlweg:generalized} gave a construction in the spirit of
Solomon's construction of descent algebras described above. Theirs is the
description of Mantaci-Reutenauer algebras used in this paper.

For a signed composition $p=(p_1, p_2, \ldots, p_k)$ of $n$ with blocks
$P_1$, \dots, $P_k$, define $W_p$ to be the subset of $W_n$ consisting of
all signed permutations $w$ such that
\[
w(P_i) \subseteq \pm P_i \ \forall \,i\in [k]\quad\text{and}\quad w(P_i)
\subseteq P_i \text{ if } p_i <0.
\]

For example, if $\lambda=(\lambda_1, \dots, \lambda_a, \lambda_{a+1}, \dots,
\lambda_{a+b})\in \CalS\CalP(n)$, then
\[
W_\lambda =W_{\Lambda_1} \dotsm W_{\Lambda_a} S_{\Lambda_{a+1}} \dotsm
S_{\Lambda_{a+b}} \cong W_{\lambda_1} \times \dotsm \times W_{\lambda_a}
\times S_{\abs { \lambda_{a+1}}} \times \dotsm \times S_{\abs
  {\lambda_{a+b}}} .
\]

In analogy with the case of symmetric groups, the subgroups $W_p$ for $p \in
\CalS\CalC(n)$ are called \emph{signed Young subgroups.} With respect to the
length function determined by the Coxeter generating set $\{t_1, s_1, \dots,
s_{n-1}\}$ of $W_n$, every left coset of $W_p$ in $W_n$ contains a unique
element of minimal length. Define $X_p$ to be the set of these minimal
length coset representatives and define
\[
x_p = \sum_{w\in X_p} w\in \CC W_n.
\]
It turns out that $\{\, x_p \mid p\in \CalS\CalC(n) \,\}$ is linearly
independent and spans a subalgebra of $\CC W_n$. This subalgebra is the
\emph{Mantaci-Reutenauer algebra of $W_n$} (see
\cite[\S2,3]{bonnafehohlweg:generalized}). In this paper, the
Mantaci-Reutenauer algebra of $W_n$ is denoted by $\Sigma(W_n)$. The reader
should be aware that this is not in accordance with the notation in
\cite{bonnafehohlweg:generalized}, where $\Sigma(W_n)$ denotes the descent
algebra of $W_n$ and $\Sigma'(W_n)$ denotes the Mantaci-Reutenauer algebra
of $W_n$.

%%%%%%%%%%%%%%%%%%%%%%%%%%%%%%%%%%%%%%%%%%%%%%%%%%%%%%%%%%%%%%%%%%%%%%%%
\subsection{Idempotents in \texorpdfstring{$\Sigma(W_n)$}
  {SWn}} \label{subsec:idemd}

Our next task is to define a complete set of primitive, orthogonal
idempotents in $\Sigma(W_n)$, and hence a complete set of orthogonal
idempotents in $\CC W_n$, that gives rise to a decomposition of the right
regular representation of $W_n$ as a direct sum of induced representations.

First, suppose $m$ is a positive integer. Recall that if $w\in S_m$ and
$i\in [m-1]$, then $i$ is a \emph{descent} of $w$ if $w(i) > w(i+1)$. Let
$D(w)$ denote the set of descents of $w$ and for $A\subseteq [m-1]$ define
$D_{\subseteq A} = \sum_{D(w) \subseteq A} w$. It is shown in \cite[Section
8.4]{reutenauer:free} that
\[
r_m= \sum_{A \subseteq [m-1] } \frac{( -1 )^ {\abs{A}}}{\abs{A}+1}
D_{\subseteq A}
\]
is an idempotent in the group algebra $\CC S_m$. In fact, by
\cite[\S3]{garsiareutenauer:decomposition}, $r_m$ lies in the descent
algebra of $S_n$, and by \cref{lem:vazlieidem} or \cite[Theorem 3.7]
{bonnafehohlweg:generalized}, $r_m$ lies in the Mantaci-Reutenauer algebra
of $W_n$. We call the idempotent $r_m$ the \emph{Reutenauer idempotent} in
$\CC S_m$.

Notice that if $P=\{z_1< \dotsm < z_m\}$ is an ordered set of positive
integers, then the Reutenauer idempotent $r_P$ is unambiguously defined in
the group algebra $\CC S_P$ by replacing the set $[m]$ by $P$ in the
preceding paragraph.

Next, define 
\[
\epsilon_m^\pm =(1/2)( \id \pm w_{0,m}),
\]
where $\id$ denotes the identity permutation in $W_m$. Then $\epsilon_m^+$
and $\epsilon_m^-$ are idempotents in $\CC W_m$, and it follows from
\cite[Example 3.5]{bonnafehohlweg:generalized} that $\epsilon_m^\pm$ is in
$\Sigma(W_m)$. Similarly, define $\epsilon_P^{\pm}=(1/2) ( \id \pm w_{0,P})$
for $P\subseteq [n]$.

Finally, suppose $p=(p_1, \dots, p_k)$ is a signed composition of $n$, and
let $\xi_i$ denote the sign of $p_i$. Define a composition $|p|$ of $n$ by
\[
|p| = ( |p_1| , \ldots, |p_k|)
\]
(note that this is non-standard notation!), and define
\[
e_p = x_{|p|} \epsilon_{P_1}^{\xi_1} r_{P_1} \cdots
\epsilon_{P_k}^{\xi_k} r_{P_k},
\]
where $x_{|p|}$ is the basis element of $\Sigma(W_n)$ corresponding to $|p|$
(note that $|p|\in \CalS\CalC(n)$). Because $\CalS\CalP(n) \subseteq
\CalS\CalC(n)$, if $\lambda$ is a signed partition of $n$, then $e_\lambda$
is defined.

\begin{prop}\label{prop:idem}
  The elements $e_p$, for $p$ a signed composition of $n$, coincide with the
  elements $I_p$ defined by Vazirani in \cite[Chapter 3]{vazirani:thesis}.
\end{prop}

This proposition is proved in \cref{sec:idem}. The next corollary follows
from the preceding proposition and \cite[\S3.7] {vazirani:thesis}.

\begin{corollary}\label{cor:ep2} 
  For a signed composition $p$ of $n$, the element $e_p$ in $\Sigma(W_n)$ is
  a quasi-idempotent with $e_p^2 = |\Stab(p)|\, e_p$. More generally, if $p$
  and $q$ are signed compositions of $n$ with $\overleftarrow p=
  \overleftarrow q$, then $e_pe_q = |\Stab(q)|\, e_q$.
\end{corollary}

%%%%%%%%%%%%%%%%%%%%%%%%%%%%%%%%%%%%%%%%%%%%%%%%%%%%%%%%%%%%%%%%%%%%%%%%
\subsection{} \label{subsec:eqdecomp}

Now suppose $\lambda = \left( \lambda_1, \ldots, \lambda_a, \lambda_{a+1},
  \ldots, \lambda_{a+b} \right)$ is a signed partition of $n$ and define
\[
E_{\lambda} = \frac{1}{(a+b)!} \sum_{\overleftarrow p = \lambda} e_p.
\]
It follows from \cref{cor:ep2} that $E_\lambda$ is an idempotent in
$\mathbb{C}W_n$, and by \cref{prop:idem}, $E_\lambda$ coincides with
$E_\lambda$ as defined by Vazirani in \cite[Chapter 4]{vazirani:thesis}, so
the set $\{\,E_\lambda \mid \lambda \in \CalS \CalP(n) \, \}$ is a complete
family of primitive, orthogonal idempotents in $\Sigma(W_n)$. Because the
$E_{\lambda}$'s form a complete set of orthogonal idempotents in
$\mathbb{C}W_n$, we have the direct sum decomposition
\begin{equation}\label{eq:decomp}
  \mathbb{C}W_n \cong \bigoplus\limits_{\lambda \in \CalS\CalP(n)}
  E_{\lambda}\mathbb{C}W_n.
\end{equation}

%%%%%%%%%%%%%%%%%%%%%%%%%%%%%%%%%%%%%%%%%%%%%%%%%%
\subsection{} \label{subsec:cent}

For $i \in [a+b]$, define the positive $|\lambda_i|$-cycle
\[
c_i=( \widehat \lambda_{i-1} +1 \;\;\; \widehat \lambda_{i-1} + 2 \; \;
\cdots \; \; \widehat \lambda_i)
\]  
and the negative $|\lambda_i|$-cycle
\[
d_i =
\begin{cases}
  c_i w_{0,\Lambda_i} & \text{if $\lambda_i$ is odd} \\
  ( \widehat \lambda_{i-1} +1 \; \; \; \widehat \lambda_{i-1} +2 \; \; \;
  \cdots \; \; \widehat \lambda_{i} )^- & \text{if $\lambda_i$ is even.}
\end{cases}
\]
Note that $c_i$ and $d_i$ are supported on the block $\Lambda_i$ of $[n]$
and are defined for both the positive and negative parts of $\lambda$.
Finally, define
\[
w_{\lambda}= c_1 \cdots c_a d_{a+1} \cdots d_{a+b}.
\] 
Then $w_{\lambda}$ is an element of $W_n$ with signed cycle type $\lambda$.
Because of our sign conventions, in general $w_\lambda$ is not in the signed
Young subgroup $W_\lambda$.

As in \cite[\S4.2]{konvalinkapfeifferroever:centralizers}, the centralizer
in $W_n$ of $w_{\lambda}$ is generated by
\[
\{\, c_i, w_{0,\Lambda_i} \mid i\in [a] \,\} \amalg \{\, d_i \mid i\in [a+b]
\setminus [a] \,\} \amalg \{\, y_i \mid \lambda_i = \lambda_{i+1}, \
i\in[a+b-1] \,\},
\] 
where for $i \in [a+b-1] $, $y_i$ is the permutation in $W_n$ defined by
\[
y_i (l) =
\begin{cases}
  l & \text{if $l\notin \Lambda_i \cup \Lambda_{i+1}$} \\
  l + \abs{\lambda_i} & \text{if $l \in \Lambda_i$} \\
  l- \abs{\lambda_i} & \text{if $l \in \Lambda_{i+1}$.}
\end{cases}
\]
Then $y_i$ fixes $[n] \setminus( \Lambda_i \cup\Lambda_{i+1})$ pointwise and
switches the blocks $\Lambda_i$ and $\Lambda_{i+1}$.

For example, if $\lambda = ( 2, 2, 1 , \overline 3, \overline{2},
\overline{2})$, then
\[
w_{\lambda} = c_1c_2c_3d_4d_5 d_6= (1 \; \; 2)(3 \; \; 4)(5) (6\;\;\overline
7\;\;8)^- (9\;\;10)^- (11\;\; 12)^-,
\]
and $Z_{W_{12}}(w_\lambda) = \langle c_1, c_2, c_3, w_{0,\Lambda_1},
w_{0,\Lambda_2}, w_{0,\Lambda_3}, d_4, d_5, d_6, y_1, y_5\rangle$, where
$c_3=\id$, $w_{0,\Lambda_1}=t_1t_2$, $w_{0,\Lambda_2}=t_3t_4$,
$w_{0,\Lambda_3}=t_3$,
\[
y_1 = \left(
  \begin{array}{cc|cc|c c }
    1 & 2 & 3 & 4 & 5 & \dotsm  \\
    3 & 4 & 1 & 2 & 5 & \dotsm 
  \end{array}
\right) \quad\text{and}\quad y_5 = \left(
  \begin{array}{ccc|cc|cc}
    1 & \dotsm & 8 & 9 & 10 & 11 &12 \\
    1 & \dotsm & 8 & 11 & 12 & 9  & 10
  \end{array} \right).
\]

%%%%%%%%%%%%%%%%%%%%%%%%%%%%%%%%%%%%%%%%%%%%%%%%%%
\subsection{} \label{subsec:thm}

For a positive integer $m$, let $\omega_m$ be the primitive $m^{th}$ root of
unity
\[
\omega_m = e^{2\pi \sqrt{-1}/m}.
\] 
Also, for a group $G$ and an element $g$ in $G$ of order $|g|=m$, define an
idempotent $\zeta_g$ in $\mathbb{C}G$ by
\[
\zeta_g = \frac{1}{m} \sum\limits_{j=1}^{m} \omega_m^{-j} g^j.
\]
If $m$ is odd, define also
\[
\tilde \zeta_g = \frac{1}{m} \sum\limits_{j=1}^{m} (\omega_m^{(m+1)/2})^{-j}
g^j.
\]
(The coefficient of $g^j$ has been chosen to simplify the formula for
$\varphi_\lambda$ in \cref{thm:main}.)

For $i \in [a+b]$, set
\[
f_i =
\begin{cases}
  \epsilon_{\Lambda_i}^{+}\zeta_{c_i} & \text{if $i\in [a]$} \\
  \epsilon_{\Lambda_i}^{-} \tilde \zeta_{c_i} & \text{if $i \in
    [a+b]\setminus [a]$ and $\lambda_i$ is odd} \\
  \zeta_{d_i} & \text{if $i \in [a+b]\setminus [a]$ and $\lambda_i$ is
    even,}
\end{cases}
\]
and define $\elt$ in $\mathbb{C}W_n$ by
\[
\elt = \xlp f_1  \cdots f_{a+b}.
\] 
We can now state the first main theorem.

\begin{theorem}\label{thm:main}
  Suppose $\lambda$ is a signed partition of $n$.
  \begin{enumerate}
  \item The group $Z_{W_n}(w_{\lambda})$ acts on $\elt$ on the right as
    scalars. Let $\varphi_{\lambda}$ be the character afforded by the
    $\mathbb{C}Z_{W_n}(w_{\lambda})$-module $\mathbb{C}\elt$. Then
    $\varphi_{\lambda}$ is given by
    \[
    \varphi_\lambda(w)=
    \begin{cases}
      \omega_{|c_i|} &\text{if $w=c_i$ for $i\in [a]$}\\
      \omega_{|d_i|}&\text {if $w=d_i$ for $i\in [a+b]\setminus [a]$}\\
      1 &\parbox[t]{.6\textwidth}{if $w=w_{0,\Lambda_i}$ for $i\in [a]$, or
        if $w=y_i$ for $i\in [a+b-1]$ with $\lambda_i=\lambda_{i+1}$.}
    \end{cases}
    \]
  \item There is an isomorphism of right $\mathbb{C}W_n$-modules
    \[
    E_{\lambda} \mathbb{C}W_n \cong \Ind_{Z_{W_n}(w_{\lambda})}^{W_n} (\CC
    \elt).
    \]
  \end{enumerate}
\end{theorem}

The theorem is proved in \cref{sec:mainthm}. A key ingredient in the proof
is \cref{prop:evenind}, where it is shown that if $\lambda_i$ is even, $C=
\la c_i, w_{0,\Lambda_i} \ra$ (the direct product of a cyclic group of order
two and a cyclic group of order $|\lambda_i|$), and $D= \la d_i\ra$ (a
cyclic group of order $2|\lambda_i|$), then $\Ind_{C}^{W_{\Lambda_i}} (\CC
\epsilon_{\Lambda_i}^{-} \zeta_{c_i} ) = \Ind_{D}^{W_{\Lambda_i}} (\CC
\zeta_{d_i} )$.

The next corollary follows immediately from
\cref{subsec:eqdecomp}\eqref{eq:decomp} and the theorem.

\begin{corollary}
  Let $\rho_n$ denote the regular character of $W_n$. Then
  \[
  \rho_n= \bigoplus_{\lambda \in \CalS\CalP(n)}
  \Ind_{Z_{W_n}(w_{\lambda})}^{W_n} (\varphi_\lambda).
  \]
\end{corollary}

%%%%%%%%%%%%%%%%%%%%%%%%%%%%%%%%%%%%%%%%%%%%%%%%%%%%%%%%%%%%%%%
\subsection{A question of Bonnaf\'e} \label{subsec:bonn}

Let $\cfC(W_n)$ denote the algebra of $\CC$-valued class functions on
$W_n$. Then $\cfC(W_n)$ is a split, semisimple, commutative $\CC$-algebra.

For $\lambda\in \CalS\CalP(n)$, let $u_\lambda$ be the characteristic
function of the conjugacy class of $w_\lambda$. Then $\{\, u_\lambda \mid
\lambda\in \CalS\CalP(n)\,\}$ is the (unique) basis of $\cfC(W_n)$
consisting of primitive idempotents.

Next, let $1_{W_p}$ be the trivial character of $W_p$ and define
\[
\theta_n\colon \Sigma(W_n)\to \cfC(W_n)\quad \text{by} \quad \theta_n(x_p)=
\Ind_{W_p}^{W_n}(1_{W_p})
\]
and linearity. Bonnaf\'e and Hohlweg \cite[Theorem 3.7]
{bonnafehohlweg:generalized} have shown that $\theta_n$ is an algebra
homomorphism with kernel equal to the Jacobson radical of
$\Sigma(W_n)$. Therefore, if $\{\, F_\lambda \mid \lambda\in
\CalS\CalP(n)\,\}$ is a complete set of primitive, orthogonal idempotents in
$\Sigma(W_n)$, then $\{\, \theta_n (F_\lambda) \mid \lambda\in
\CalS\CalP(n)\,\}$ is the set of primitive idempotents in $\cfC(W_n)$, and
so there is a permutation of $\CalS\CalP(n)$, say $\lambda\mapsto
\lambda^*$, so that $\theta_n(F_\lambda)= u_{\lambda^*}$ for all $\lambda
\in \CalS\CalP(n)$.

Bonnaf\'e \cite[\S10]{bonnafe:mantaci} asked whether it was possible to find
a set of primitive idempotents $\{\,F_\lambda\,\}$ such that $F_\lambda \CC
W_n \cong \Ind_{Z_{W_n}(w_\lambda)} ^{W_n}(\eta_\lambda)$ for some linear
character $\eta_\lambda$ of $Z_{W_n}(w_\lambda)$. It follows from
\cref{thm:main} that the idempotents $\{\, E_\lambda\mid \lambda\in
\CalS\CalP(n)\,\}$ constructed by Vazirani give a positive answer to this
question. The permutation $\lambda \mapsto \lambda^*$ such that
$\theta_n(E_\lambda)= u_{\lambda^*}$ is given in the next theorem.

For $p= (p_1, \dots, p_{k}) \in \CalS\CalC(n)$ define $p'\in \CalS\CalC(n)$
by
\[
p_i' = \begin{cases} p_i &\text{if $p_i$ is odd} \\
  \overline{p_i} &\text{if $p_i$ is even.}
\end{cases}
\]
For example, if $\lambda=(4,3,2,2,1, \overline 6, \overline 5, \overline 4,
\overline 3, \overline 3, \overline 2)$, then
\[
\lambda'= (\overline 4,3,\overline 2,\overline 2,1, 6, \overline 5, 4,
\overline 3, \overline 3, 2)\quad\text{and}\quad\overleftarrow{\lambda'}=
(6,4,3,2,1, \overline 5, \overline 4, \overline 3, \overline 3, \overline
2,\overline 2 ).
\]

\begin{theorem}\label{thm:bonnperm}
  Suppose $\lambda$ is a signed partition of $n$. Then $\theta_n(E_\lambda)=
  u_{\overleftarrow{\lambda'}}$.
\end{theorem}

This theorem is proved in \cref{sec:theta}.

%%%%%%%%%%%%%%%%%%%%%%%%%%%%%%%%%%%%%%%%%%%%%%%%%%%%%%%%%%%%%%%
%%%% Section 3 %%%%%%%%%%%
%%%%%%%%%%%%%%%%%%%%%%%%%%%%%%%%%%%%%%%%%%%%%%%%%%%%%%%%%%%%%%%
\section{Proof of \texorpdfstring{\cref{prop:idem}}{cref}} \label{sec:idem}

In this section, $p=(p_1, \dots, p_k)$ is a fixed signed composition of $n$,
and we show that the idempotent denoted by $I_p$ in \cite[Chapter
3]{vazirani:thesis} is equal to $e_p = x_{|p|} \epsilon_{P_1}^{\xi_1}
r_{P_1} \cdots \epsilon_{P_k}^{\xi_k} r_{P_k}$ (here $\xi_i$ is the sign of
$p_i$). In order to do so, we first reformulate the definition of the
idempotents in $\CC S_n$ denoted by $I_p$ in
\cite{garsiareutenauer:decomposition} (for $p\in \CalC(n)$). This requires
the basis of $\Sigma(W_n)$ used in \cite{mantacireutenauer:generalization}
and \cite{vazirani:thesis}.

%%%%%%%%%%%%%%%%%%%%%%%%%%%%%%%%%%%%%%%%%%%%%%%%%%%%%%%%%%%%%%%
\subsection{The Mantaci-Reutenauer basis of \texorpdfstring{$\Sigma(W_n)$}
  {SWn}} \label{subsec:mrbasis}

Define a partial order on $\CalS\CalC(n)$ by ``signed refinement,'' that is,
for $p, q\in \CalS\CalC(n)$, define $p\leq q$ if $q$ can be obtained from
$p$ by combining consecutive parts with the same sign. In this case, say
that $p$ is \emph{finer} than $q$. Here we are following the presentation in
\cite[Chapter 3]{vazirani:thesis}, in which the partial order is reversed
from that in \cite{mantacireutenauer:generalization}.

For example, $p = \left( 1, 1, \overline{2}, \overline{1}, 2, 3, 2 \right)$
is finer than $q = \left( 2, \overline{3}, 2, 3, 2 \right)$, which in turn
is finer than $r = \left( 2, \overline{3}, 7 \right)$; moreover, $r$ is
maximal with respect to the partial order.

Next, for $w\in W_n$, say that $i\in [n-1]$ is a \emph{descent} of $w$ if
\begin{itemize}
\item $w(i)$ and $w(i+1)$ have the same sign and $\abs{w(i)}> \abs{w(i+1)}$,
  or
\item $w(i)$ and $w(i+1)$ have opposite signs.
\end{itemize}
Let $\D(w)$ denote the set of descents of $w$. Notice that for $w\in S_n$
this definition agrees with that in \cref{subsec:idemd}, but that in general
$\D(w)$ is not the descent set of $w$ with respect to a positive system of
roots.

For example, in one row notation let $w=w(1) w(2) \cdots w(n) =
2\,1\,\overline 3\, \overline 6\, \overline 5\, 4\, 8\, \overline{7}$ in
$W_8$.  Then $\D(w)= \{1, 2, 4, 5, 7\}$, where the descents at $2$, $5$, and
$7$ arise from sign changes. Note that the descents of $w$ partition the set
$[8]$ into six blocks:
\[
w=\begin{array}{c|c|cc|c|cc|c}
  %% 1 & 2 & 3 & 4 & 5 & 6 & 7 & 8 \\
    2 & 1 & \overline{3} 
               & \overline{6} & \overline{5} & 4 & 8 & \overline{7} .
\end{array}  
\]

Finally, for $w\in W_n$, the \emph{descent shape of $w$}, denoted by
$\ds(w)$, is the signed composition $p= (p_1, p_2, \ldots, p_k)$ such that
\begin{itemize}
\item $\D(w) = \{\widehat p_1, \widehat p_2, \dots, \widehat p_{k-1} \}$,
  and
\item the sign of $p_i$ is equal to the sign of $w( \widehat p_i)$.
\end{itemize}
In other words, the descent shape of $w$ is found by using the descents of
$w$ to break the set $[n]$ into blocks, and the sizes and signs of the
blocks determine the parts of $\ds(w)$. For example, with $w$ as above,
$\ds(w)= \left( 1, 1, \overline{2}, \overline{1}, 2, \overline{1} \right)$.

Now for $p\in\CalS \CalC(n)$, define
\[
x_p^v = \sum_{\ds(w) \geq p} w.
\]
Then $\{\, x_p^v \mid p\in \CalS\CalC(n)\,\}$ is a basis of $\Sigma(W_n)$
(see \cite[\S2.8] {bonnafehohlweg:generalized}).

%%%%%%%%%%%%%%%%%%%%%%%%%%%%%%%%%%%%%%%%%%%%%%%%%%%%%%%%%%%%%%%
\subsection{}\label{subsec:psi} 

Notice that for $w\in W_n$, the following statements are equivalent:
\begin{itemize}
\item $w\in S_n$,
\item $\ds(w) \in \CalC(n)$, and
\item $\ds(w) \leq (\,n\,)$.
\end{itemize}

For a composition or signed composition $p$ of $n$, let $k(p)$ denote the
number of parts of $p$. Also, let $\psi$ denote the bijection between
(unsigned) compositions of $n$ and subsets of $[n-1]$ given by
\[
\psi (p) = \{ \widehat p_1, \widehat p_2, \dots, \widehat p_{k-1} \}
\]
when $p = ( p_1, p_2, \dots, p_k)$. The following statements follow
immediately from the definitions:
\begin{itemize}
\item For $p\in \CalC(n)$ and $w\in W_n$, $\ds(w) = p$ if and only if $w\in
  S_n$ and $D(w)= \psi(p)$.
\item For $p,q\in \CalC(n)$, $p\leq q$ if and only if $\psi(q) \subseteq
  \psi(p)$.
\item For $p\in \CalC(n)$, $k(p) = |\psi(p)|+1$.
\end{itemize}
It follows from the first two statements that for $p\in \CalC(n)$,
\[
x_p^v = \sum_{\ds(w) \geq p} w = \sum_{\substack{w\in S_n\\ D(w) \subseteq
    \psi(p)}} w = D_{\subseteq \psi(p)} .
\]

%%%%%%%%%%%%%%%%%%%%%%%%%%%%%%%%%%%%%%%%%%%%%%%%%%%%%%%%%%%%%%%
\begin{lemma}\label{lem:vazlieidem} 
  Let $m$ be a positive integer. Then the Reutenauer idempotent $r_m\in \CC
  S_m$ may be expressed as
  \[
  r_m= \sum_{p\in \CalC(m)} \frac{( -1) ^{k(p)-1}}{k(p)} x_p^v.
  \]
\end{lemma}

\begin{proof}
  Using the assertions in \cref{subsec:psi} we have
  \[
  r_m= \sum_{A \subseteq [m-1]} \frac{(-1)^ {|A|}}{|A|+1} D_{\subseteq A} =
  \sum_{p\in \CalC(m)} \frac{(-1)^{|\psi(p)|}}{|\psi(p)|+1} D_{\subseteq
    \psi(p)} = \sum_{p\in \CalC(m)} \frac{(-1)^{k(p)-1}}{k(p)} x_p^v.
  \]
\end{proof}

%%%%%%%%%%%%%%%%%%%%%%%%%%%%%%%%%%%%%%%%%%%%%%%%%%%%%%%%%%%%%%%
\subsection{Garsia-Reutenauer idempotents}\label{subsec:gridem}

Suppose $p=(p_1, \dots, p_k)$ is a composition of $n$. Garsia and Reutenauer
\cite[(3.17)]{garsiareutenauer:decomposition} define a quasi-idempotent $I_p
\in \CC S_n$ by
\begin{equation}
  \label{eq:ipm}
  I_p = \sum_{\substack{ J_1+ \dotsm + J_k = [n] \\
      |J_i| = p_i }} \rho_{[J_1]}* \dotsm *\rho_{[J_k]},  
\end{equation}
where (using the notation in \cite{garsiareutenauer:decomposition})
\begin{itemize}
\item the sum is over all ordered set partitions $J_1, \dots, J_k$ of $[n]$
  such that $|J_i|=p_i$ for $i\in [k]$,
\item if $J=\{j_1< \dots <j_m\}\subseteq [n]$, then $\rho_{[J]} = w_J r_m$,
  where $r_m$ is the Reutenauer idempotent in $\CC S_m$ (considered as a
  subalgebra of $\CC S_n$) and $w_J$ is the function from $[m]$ to $J$ given
  in two row notation by
    \[
    w_J= \left( 
      \begin{array}{cccc}
        1& 2& \cdots & m \\
        j_1 & j_2& \cdots & j_m 
      \end{array}\right) ,
    \]
    and
  \item $*$ is the concatenation product.
\end{itemize}
In the second bullet point, the equality $\rho_{[J]} = w_J r_m$ uses the
formulation of the Reutenauer idempotent in \cref{lem:vazlieidem}.

Consider a summand
\[
\rho_{[J_1]}* \dotsm *\rho_{[J_k]}= w_{J_1} r_{p_1} * \dotsm *w_{J_k}
r_{p_k}.
\]
It is straightforward to check that
\begin{equation}
  \label{eq:cat}
  w_{J_1} r_{p_1} * \dotsm *w_{J_k} r_{p_k} = w_{(J_1, \dots, J_k)} r_{P_1}
  \dots r_{P_k},  
\end{equation}
where as above the product on the right-hand side is the usual
multiplication in the group algebra $\CC S_n$ and $w_{(J_1, \dots, J_k)}$ is
the permutation given in two row notation by
\begin{equation}
  \label{eq:min}
  w_{(J_1, \dots, J_k)} = \left(
    \begin{array}{cccc}
      P_1& P_2& \cdots & P_k \\
      J_1 & J_2& \cdots & J_k 
    \end{array}\right)   ,
\end{equation}
with the convention that entries in $P_i$ and $J_i$ are written in
increasing order.

For example, suppose $p=(a,b)$, so $P_1=\{1, \dots, a\}$ and $P_2= \{a+1,
\dots, a+b\}$. Consider $w_{J} x*w_{K} y$, with $J=\{j_1<\dotsm < j_a\}$,
$K=\{k_1<\dotsm <k_b\}$, $x\in S_a$, and $y\in S_b$. Say $x=x_1\; x_2\;
\cdots \;x_a$ and $y=y_1\; y_2\; \cdots \;y_b$ in one row notation. Using
two row ``block'' notation for permutations, write
\[
w_J x= \left( \begin{array}{c} [a] \\ J \end{array}\right)
\left( \begin{array}{c} [a] \\ x \end{array}\right) =
\left( \begin{array}{c} [a] \\ J' \end{array}\right) \quad\text{and}\quad
w_Ky= \left( \begin{array}{c} [b] \\ K \end{array}\right)
\left( \begin{array}{c} [b] \\ y \end{array}\right) =
\left( \begin{array}{c} [b] \\ K' \end{array}\right) ,
\]
where now $J'=j_1'\; \dotsm \; j_a'$ is obtained from $J= j_1\; \dotsm \;
j_a$ by permuting the entries, and similarly for $K'$. Then
\begin{align*}
  w_{J} x*w_{K} y 
  &=\left( \begin{array}{c} [a] \\ J' \end{array}\right) *
  \left( \begin{array}{c} [b] \\ K' \end{array}\right)\\ 
  &=\left( \begin{array}{cc} P_1&P_2 \\ J' &P_2\end{array}\right) \cdot
  \left( \begin{array}{cc} P_1&P_2\\ P_1&a+K' \end{array}\right) \\ 
  &=\left( \begin{array}{cc} P_1&P_2 \\ J &P_2\end{array}\right)
  \left( \begin{array}{cc} P_1&P_2 \\ x &P_2\end{array}\right) \cdot
  \left( \begin{array}{cc} P_1&P_2 \\ P_1 &a+K\end{array}\right)
  \left( \begin{array}{cc} P_1&P_2 \\ P_1 &a+y\end{array}\right) \\
  &=\left( \begin{array}{cc} P_1&P_2 \\ J &P_2\end{array}\right)
  \left( \begin{array}{cc} P_1&P_2 \\ P_1 &a+K\end{array}\right) \cdot
  \left( \begin{array}{cc} P_1&P_2 \\ x &P_2\end{array}\right)
  \left( \begin{array}{cc} P_1&P_2 \\ P_1 &a+y\end{array}\right) \\
  &=\left( \begin{array}{cc} P_1&P_2 \\ J &a+K\end{array}\right)
  \left( \begin{array}{cc} P_1&P_2 \\ x &a+y\end{array}\right)\\ 
  &=w_{(J,K)} \left( \begin{array}{cc} P_1&P_2 \\ x &P_2 \end{array}\right)
  \left( \begin{array}{cc} P_1&P_2 \\ P_1 &a+y \end{array}\right),
\end{align*}
where $a+K'$, $a+K$, and $a+y$ denotes adding $a$ to each entry of $K'$,
$K$, and $y$, respectively.

It follows from the definitions (see \cite[Remark 2.1]
{bonnafehohlweg:generalized}) that if $w\in W_n$, then $w\in X_p$ if and
only if $w(j)>0$ for $j\in [n]$, and $w|_{P_l}\colon P_l\to [n]$ is
increasing for $l\in [k]$. Thus $w_{(J_1, \dots, J_k)}\in X_p$ and in fact
\[
X_p= \{\, w_{(J_1, \dots, J_k)} \mid \text{$J_1+ \dotsm + J_k = [n]$ and
  $\forall i\in [k],\, |J_i| = p_i$} \,\} .
\]

Putting the pieces together gives
\begin{equation*}
  \label{eq:gridem}
  I_p = \sum_{\substack{ J_1+ \dotsm + J_k = [n] \\ |J_i| = p_i }}
  \rho_{[J_1]}* \dotsm *\rho_{[J_k]} =\sum_{w\in X_p} w r_{P_1} \dotsm
  r_{P_k}= x_p  r_{P_1} \dotsm r_{P_k}.  
\end{equation*}

Garsia and Reutenauer \cite[\S3, \S4]{garsiareutenauer:decomposition} prove
the remarkable fact that $\{\, I_p\mid p\in \CalC(n)\,\}$ is a basis of the
descent algebra of $S_n$ consisting of quasi-idempotents.

%%%%%%%%%%%%%%%%%%%%%%%%%%%%%%%%%%%%%%%%%%%%%%%%%%%%%%%%%%%%%%%
\subsection{Vazirani's idempotents} \label{subsec:vazmr}

Vazirani \cite{vazirani:thesis} extends the constructions of Garsia and
Reutenauer to $\Sigma(W_n)$.

Suppose $m$ is a positive integer, and set
\[
I_{(m)}^{\pm} = \epsilon_m^{\pm} \cdot r_m \in \CC W_m .
\]

Now, given a signed composition $p=(p_1, \dots, p_k)$ of $n$, in analogy
with \cref{subsec:gridem}\eqref{eq:ipm} define
\[ 
I_p = \sum_{\substack{ J_1+ \dots + J_k =[n]\\
    |J_i|= |p_i| }} I_{[J_1 ]}^{\xi_1}* \cdots * I_{\left[J_k
  \right]}^{\xi_k},
\]
where $\xi_i$ is the sign of $p_i$,
\begin{itemize}
\item the sum is over all ordered set partitions $\{J_1, \dots, J_k\}$ of
  $[n]$ with $|J_i|=|p_i|$ for $i\in [k]$,
\item if $J=\{j_1< \dots <j_m\}\subseteq [n]$, then 
  \[
  I_{[J]}^\xi = w_J I_{(m)}^\xi = w_J \,\epsilon_m^{\xi}\, r_m,
  \]
  where $w_J$ is as in \cref{subsec:gridem} and $\xi\in \{+,-\}$, and
\item $*$ is the concatenation product.
\end{itemize}

Consider a summand
\[
I_{[J_1 ]}^{\xi_1}* \cdots * I_{\left[J_k \right]}^{\xi_k} = w_{J_1}
\epsilon_{|p_1|}^{\xi_1} r_{|p_1|} * \dotsm *w_{J_k} \epsilon_{|p_k|}
^{\xi_k} r_{|p_k|}.
\]
Substituting $\epsilon_{|p_i|}^{\xi_i} = (1/2) (\id \pm w_{0,|p_i|})$,
expanding the right-hand side, using the computation of the concatenation
product in \cref{subsec:gridem}\eqref{eq:cat}, and then simplifying the
expression using the definition of $\epsilon_{|p_i|}^{\xi_i}$ again shows
that
\[
w_{J_1} \epsilon_{|p_1|}^{\xi_1} r_{|p_1|} * \dotsm *w_{J_k}
\epsilon_{|p_k|} ^{\xi_k} r_{|p_k|} = w_{(J_1, \dots, J_k)}
\epsilon_{P_1}^{\xi_1} r_{P_1} \dotsm \epsilon_{P_1}^{\xi_k} r_{P_k},
\]
where the product on the right-hand side is the usual multiplication in the
group algebra and $w_{(J_1, \dots, J_k)}$ is the permutation in
\cref{subsec:gridem}\eqref{eq:min}.

Putting the pieces together this time gives
\begin{equation*}
  \label{eq:gridem*}
  I_p = \sum_{\substack{ J_1+ \dotsm + J_k = [n] \\ |J_i| = |p_i| }} I_{[J_1
    ]}^{\xi_1}* \cdots * I_{\left[J_k \right]}^{\xi_k} =\sum_{w\in X_{|p|}}
  w\, \epsilon_{P_1}^{\xi_1} r_{P_1} \dotsm \epsilon_{P_k}^{\xi_k} r_{P_k}
  = x_{|p|} \epsilon_{P_1}^{\xi_1} r_{P_1} \dotsm \epsilon_{P_k}^{\xi_k}
  r_{P_k} ,
\end{equation*}
which is the assertion in the proposition.

%%%%%%%%%%%%%%%%%%%%%%%%%%%%%%%%%%%%%%%%%%%%%%%%%%%%%%%%%%%%%%%
%%%% Section 4 %%%%%%%%%%%
%%%%%%%%%%%%%%%%%%%%%%%%%%%%%%%%%%%%%%%%%%%%%%%%%%%%%%%%%%%%%%%
\section{Computation of some induced characters}\label{sec:cycchar}

In this section, $m$ is a positive integer, and we consider the
hyperoctahedral group $W_m$. Let $c$ be the positive $m$-cycle $c=\left( 1
  \; 2\; \cdots \; m \right)$, let $d$ be the negative $m$-cycle $d=
\left( 1 \; 2 \; \cdots \; m\right)^-$, and set $w_0 = w_{0,m}$. In preparation for the proof of
\cref{thm:main}, we compare characters induced from $\la d \ra$ and $\la c,
w_0 \ra$ to $W_m$, when $m$ is even. This section is devoted to the proof of
\cref{prop:evenind}, which asserts that when $m$ is even, suitably chosen
characters of $\la d \ra$ and $\la c, w_0 \ra$ induce to the same character
of $W_m$.

Recall that for a group $G$ and an element $g\in G$ of order $m$, the
idempotent $\zeta_g$ in $\mathbb{C}\la g \ra$ is defined by $\zeta_g =
\frac{1}{m} \sum_{i=1}^{m} \omega_m^{-i}g^i$. Throughout this section
$\epsilon = \frac{1}{2} \left( \id - w_0 \right)$. Then $\epsilon$ and
$\zeta_c$ are commuting idempotents in $\mathbb{C} \la c, w_0 \ra$. Let
$\chi_{\epsilon \zeta_c}$ be the character of $\la c, w_0 \ra$ afforded by
the right ideal $\epsilon \zeta_c \mathbb{C}\la c, w_0 \ra$, and let
$\chi_{\zeta_d}$ be the character of $\la d \ra$ afforded by the right ideal
$\zeta_d \mathbb{C} \la d \ra$.

\begin{prop}\label{prop:evenind}
  Suppose $m$ is even. Then there is an isomorphism of right $\CC
  W_m$-modules
  \[
  \epsilon \zeta_c\mathbb{C}W_m \cong \zeta_d\mathbb{C}W_m.
  \]
\end{prop}

By \cite[11.21]{curtisreiner:methodsI}, $\Ind_{\langle c, w_0 \rangle
}^{W_m} \left( \chi_{\epsilon \zeta_c} \right)$ is the character of the
representation of $W_m$ acting on the right ideal $\epsilon
\zeta_c\mathbb{C}W_m$, and $\Ind_{\langle d\rangle }^{W_m} \left(
  \chi_{\zeta_d} \right)$ is the character of the representation of $W_m$
acting on the right ideal $\zeta_d\mathbb{C}W_m$. Thus, to prove the
proposition it is enough to show that
\begin{equation}
  \label{eq:ind1}
  \Ind_{\la c, w_0 \ra}^{W_m}\left(\chi_{\epsilon \zeta_c} \right)
  =\Ind_{\langle d\rangle }^{W_m} \left( \chi_{\zeta_d} \right) .
\end{equation}
We prove \eqref{eq:ind1} by showing that both induced characters take the
same values on all conjugacy classes. In order to do so we need some
preliminary lemmas. The first lemma is due to Littlewood (see \cite[Exercise
9.16] {curtisreiner:methodsI} for the left-sided version).

\begin{lemma}\label{lem:idemchar} 
  For a finite group $G$ and an idempotent $e = \sum_{g_1 \in G}
  \gamma_{g_1} g_1$ in $\mathbb{C}G$, the character $\chi_e$ of $G$ afforded
  by $e\mathbb{C}G$ is given by
  \[
  \chi_e(g) = \abs{Z_G(g)} \sum_{g_1 \in \ccl(g)} \gamma_{g_1},
  \]
  where $\ccl{(g)}$ is the conjugacy class of $g$.
\end{lemma}

%%%%%%%%%%%%%%%%%%%%%%%%%%%%%%%%%%%%%%%%%%%%%%%%%%
\subsection{} \label{ssec:Eqncd}

It follows from \cite[11.21]{curtisreiner:methodsI} and the lemma that for
$g$ in $W_m$,
\begin{equation*}\label{eq:c}
  \Ind_{\la c, w_0 \ra}^{W_m} \left(\chi_{\epsilon \zeta_c}\right) (g) =
  \abs{Z_{W_m} (g)} \sum_{g_1 \in \ccl(g)} \widetilde
  \gamma_{g_1}, 
\end{equation*}
where
\[
\widetilde \gamma_{g_1} = \begin{cases}
  \frac{1}{2m} \omega_m^{-j} & \text{if $g_1 = c^j$ for
    some $j\in [m]$} \\
  -\frac{1}{2m} \omega_m^{-j} & \text{if $g_1=w_0 c^j$ for
    some $j\in [m]$}\\
  0 & \text{ otherwise,}
\end{cases}
\]
and
\begin{equation*}\label{eq:d}
  \Ind_{\langle d\rangle }^{W_m} \left( \chi_{\zeta_d} \right)(g) =
  \abs{Z_{W_m} (g)} \sum_{g_1 \in \ccl(g)} \gamma_{g_1},
\end{equation*}
where
\[
\gamma_{g_1} = \begin{cases} \frac{1}{2m} \omega_{2m}^{-j} & \text{if
    $g_1=d^j$ for some $j\in [2m]$} \\
  0 & \text{otherwise.}
\end{cases}
\]

To continue, we need to know the signed cycle types of the elements that
appear in the subgroups $\langle c, w_0\rangle$ and $\langle d \rangle$ of
$W_m$. The proof of the next lemma is straightforward and is omitted.

\begin{lemma}\label{lem:ct}
  Suppose that $j \in [m]$. Set $\ell = \gcd(m,j)$, $a= m/\ell$, and
  $b=j/\ell$.
  \begin{enumerate}
  \item The cycle type of $c^j$ is $(a^{\ell})$, i.e., $\ell$-many cycles of
    length $a$.
  \item The cycle type of $w_0c^j$ is
    \[
    \begin{cases}
      (a^{\ell}) &\text{if $a$ is even} \\
      (\overline a^{\ell}) &\text{if $a$ is odd,}
    \end{cases}
    \]
    where $(\overline a^{\ell})$ denotes $\ell$-many negative $a$-cycles.
  \item The cycle type of $d^j$ is
    \[
    \begin{cases}
      (a^{\ell}) &\text{if $b$ is even} \\
      (\overline a^{\ell}) &\text{if $b$ is odd.}
    \end{cases}
    \]
  \item The cycle type of $d^{m+j}$ is
    \[
    \begin{cases}
      (\overline{a}^\ell) &\text{if $a$ is even}\\
      (a^{\ell}) &\text{if $a$ is odd and $b$ is odd} \\
      (\overline a^{\ell}) &\text{if $a$ is odd and $b$ is even.}
    \end{cases}
    \]
  \end{enumerate}
  In particular, the cycle types that occur in $\la c, w_0\ra$ and $\la
  d\ra$ are the cycle types $(a^\ell)$ and $(\overline a^\ell)$, where
  $\ell$ divides $m$ and $a = m/\ell$.
\end{lemma}

%%%%%%%%%%%%%%%%%%%%%%%%%%%%%%%%%%%%%%%%%%%%%%%%%%
\subsection{}\label{subsec:charnot}

With the computation of the cycle types and \cref{lem:idemchar} in hand we
can compute $\Ind_{\la c, w_0 \ra}^{W_m}\left(\chi_{\epsilon \zeta_c}
\right)$ and $\Ind_{\la d \ra}^{W_m} \left( \chi_{\zeta_d} \right)$. The
proof of \cref{prop:evenind} follows from the next lemma.

For $\lambda\in \CalS\CalP(m)$, denote the value of a character $\chi$ on
the conjugacy class with cycle type $\lambda$ by $(\chi)_\lambda$.

\begin{lemma}\label{lem:aodd}
  Suppose $\lambda\in \CalS\CalP(m)$. Then
  \begin{multline*}
    \left( \Ind_{\la c, w_0 \ra}^{W_m} (\chi_{\epsilon \zeta_c})
    \right)_{\lambda} = \left( \Ind_{\la d \ra}^{W_m} (\chi_{\zeta_d})
    \right)_{\lambda}\\ = \begin{cases} a^{\ell-1} (\ell-1)!\, 2^{\ell-1}
      \mu(a) &\text{if $m=a\ell$, $a$ is odd, and $\lambda= (a^\ell)$} \\
      - a^{\ell-1} (\ell-1)!\, 2^{\ell-1} \mu(a) & \text{if $m=a\ell$, $a$
        is odd, and $\lambda= (\overline a^\ell)$}\\
      0 &\text{otherwise,}
    \end{cases}
  \end{multline*}
  where $\mu$ is the M\"obius function.
\end{lemma}

\begin{proof}
  The strategy of the proof is to compute the values of both induced
  characters and then observe that they are equal. We give the details for
  the character $\Ind_{\la d \ra}^{W_m} (\chi_{\zeta_d})$. The computation
  of the character values of $\Ind_{\la c, w_0 \ra}^{W_m} (\chi_{\epsilon
    \zeta_c})$ is similar (and easier), and is omitted.

  Now consider $\left( \Ind_{\la d\ra}^{W_m} (\chi_{\zeta_d})
  \right)_{\lambda}$.  It follows from \cref{lem:ct} that if $\lambda$ is
  not equal to $(a^\ell)$ or $(\overline a^\ell)$ for some factorization
  $m=a\ell$, then $\left( \Ind_{\la d\ra}^{W_m} (\chi_{\zeta_d})
  \right)_{\lambda}=0$.

  Suppose $m=a\ell$ and $\lambda=(a^\ell)$ or $(\overline a^\ell)$. There
  are two cases depending on whether $a$ is even or odd.

  First suppose that $a$ is even. By \cref{lem:ct}, no elements in $\la d
  \ra$ have signed cycle type $(a^\ell )$, and the elements with signed
  cycle type $( \overline a^\ell )$ are the elements $d^j$ and $d^{m+j}$
  with $j\in [m]$ and $\gcd(m,j) = \ell$. Using the notation in
  \cref{ssec:Eqncd}, for each such $j$ we can pair the elements $d^j$ and
  $d^{m+j}$ to obtain
  \[
  \gamma_{d^j} + \gamma_{d^{m+j}} = ( \omega_{2m}^{-j} +
  \omega_{2m}^{-(m+j)} )/ 2m = ( \omega_{2m}^{-j} - \omega_{2m}^{-j}) /2m =
  0.
  \]
  Therefore, $\Ind_{\la d \ra}^{W_m} (\chi_{\zeta_d})$ vanishes on the
  conjugacy classes of $W_m$ with signed cycle type $(a^\ell)$ and
  $(\overline a^\ell)$ when $a$ is even.

  Now suppose that $m=a\ell$ and $a$ is odd. There are four subcases.

  First, the only element in $\la d \ra$ with signed cycle type
  $\lambda=(1^m)$ is the identity, and
  \[
  \left(\Ind_{\la d \ra}^{W_m} (\chi_{\zeta_d})\right) _{(1^m)} =|W_m|/2m
  =2^{m-1}(m-1)! =a^{\ell-1} (\ell-1)!\, 2^{\ell-1} \mu(a),
  \]
  because $\ell =m$ and $\mu(1)=1$.  Similarly, the only element in $\la
  d\ra$ with signed cycle type $\lambda=(\overline 1^m)$ is $w_0=d^m$, and
  \[
  \left(\Ind_{\la d \ra}^{W_m} (\chi_{\zeta_d})\right) _{(\overline 1^m)}
  =|W_m| \, \omega_{2m}^{-m} /2m = -2^{m-1}(m-1)! =-a^{\ell-1} (\ell-1)!\,
  2^{\ell-1} \mu(a).
  \]

  Now suppose that $a>1$ and $\lambda=(a^\ell)$. Set
  \[
  A= \{\,j\in [m-1] \mid \text{$\gcd(m,j)= \ell$ and $j/\ell$ is even}\,\}.
  \]
  Because $d^j$ and $(d^j)^{-1} = d^{2m-j}$ are conjugate in $W_m$, the set
  of elements in $\la d \ra$ with signed cycle type $(a^\ell)$ is $\{\, d^j,
  d^{2m-j}\mid j\in A\,\}$. Note that because $m$ is even and $a$ is odd,
  $\ell =m/a$ is even, and therefore if $j\in A$, then $\ell$ divides $j$
  and so $j$ is even. Thus, using the fact that $c^\ell$ has signed cycle
  type $(a^\ell)$ we have
  \begin{equation}
    \label{eq:aaa}
    \begin{aligned}
      \left(\Ind_{\la d \ra}^{W_m} (\chi_{\zeta_d})\right) _{(a^{\ell})} & =
      |Z_{W_m} (c^{\ell})| \sum_{j\in A} ( \omega_{2m}^{-j} +
      \omega_{2m}^{-(2m-j)}) /2m \\
      &= a^{\ell-1} (\ell-1)!\, 2^{\ell-1} \sum_{j\in A} (
      \omega_{2m}^{-2(j/2)} +  \omega_{2m}^{-(2m-2(j/2))} ) \\
      &= a^{\ell-1} (\ell-1)!\, 2^{\ell-1} \sum_{k\in A/2} (\omega_{m}^{-k}
      + \omega_m^{-(m-k)} ).
    \end{aligned}
  \end{equation}

  Observe that with our assumptions on $m$ and $a$, if $k=j/2\in A/2$, then
  $\gcd(m,k) = \gcd(m, j) =\ell$. Moreover, $\gcd(m,k) = \gcd(m,m-k)$, so
  $\gcd(m, m-k)=\ell$. Therefore,
  \[
  \{\,k, m-k \mid k \in A/2\,\} = \{\,h\in [m]\mid \gcd(m,h)=\ell\,\},
  \]
  and so letting $k'= h/\ell$, and using that $m/\ell =a$, we have
  \[
  \{\, h\in [m] \mid \gcd(m,h) = \ell \,\} = \{ \ell k' \mid
  \text{$k'\in[a]$ and $\gcd(a, k') = 1$} \,\}.
  \] 
  Thus
  \begin{equation}
    \label{eq:mu}
    \sum_{k\in A/2} (\omega_{m}^{-k} + \omega_m^{-(m-k)} ) =
    \sum_{\substack{h\in [m]\\ \gcd(m,h)=\ell}} \omega_m^{-h} =
    \sum_{\substack{k' \in [a] \\ \gcd (a, k') = 1}} \omega_m^{-\ell k'} =
    \sum_{\substack{ k'\in [a] \\ \gcd (a, k') = 1}} \omega_a^{-k'} = \mu(a),    
  \end{equation}
  where the last equality holds because the sum is over all primitive
  $a^{\text{th}}$ roots of unity. Substituting \eqref{eq:mu} in
  \eqref{eq:aaa} gives $\left(\Ind_{\la d \ra}^{W_m}
    (\chi_{\zeta_d})\right)_{(a^{\ell})} = a^{\ell-1} (\ell-1)!\, 2^{\ell-1}
  \mu(a)$.
  
  Finally suppose that $a>1$ and $\lambda=(\overline a^\ell)$.  Set
  \[
  B= \{\,j\in [m-1] \mid \text{$\gcd(m,j)= \ell$ and $j/\ell$ is odd}\,\}.
  \]
  Then the set of elements in $\la d \ra$ with signed cycle type $(\overline
  a^\ell)$ is $\{\, d^j, d^{2m-j}\mid j\in B\,\}$, and because $w_0c^\ell$
  has signed cycle type $(\overline a^\ell)$, we have
  \begin{align*}
    \left(\Ind_{\la d \ra}^{W_m} (\chi_{\zeta_d})\right) _{(\overline
    a^{\ell})}  
    & = |Z_{W_m} (w_0c^{\ell})| \sum_{j\in B} ( \omega_{2m}^{-j} +
      \omega_{2m}^{-(2m-j)}) /2m\\
    & = a^{\ell-1} (\ell-1)!\, 2^{\ell-1} \sum_{j\in B} ( \omega_{2m}^{-j} +
      \omega_{2m}^{-(2m-j)}) .
  \end{align*}

  Again, if $j\in B$, then $j$ is even. Hence $m+j$ and $m-j$ are both even,
  so
  \[
  \omega_{2m}^{-j} + \omega_{2m}^{-(2m-j)} = \omega_{2m}^m
  (\omega_{2m}^{m-j} + \omega_{2m}^{m+j}) = - (\omega_{2m}^{m-j} +
  \omega_{2m}^{m+j}) = -( \omega_m^{\frac{m-j}{2}} +
  \omega_m^{\frac{m+j}{2}}).
  \]

  We now show that if $j\in B$, then $\ell$ divides $(m+ j)/2$ and
  $(m-j)/2$, from which it follows that $\gcd(m, (m\pm j)/2) = \ell$, and
  hence that
  \[
  \{\, (m\pm j)/2 \mid j\in B\,\} = \{\, h\in [m] \mid \gcd(m,h)=\ell \,\}.
  \]
  Say $m=2^{k_m} q_m$, $\ell = 2^{k_{\ell}}q_{\ell}$, and $j=2^{k_j}q_j$,
  where $q_m$, $q_\ell$, and $q_j$ are all odd. Since $\ell|m$ and $\ell|j$,
  and $m/\ell = a$ and $j/\ell$ are odd, it must be that $k_m = k_{\ell}$
  and $k_j = k_{\ell}$. Set $k=k_{\ell} = k_j = k_m$. Then $m\pm j = 2^{k}
  (q_m \pm q_j)$. Because $q_m$ and $q_j$ are odd, $q_m \pm q_j$ is
  even. Hence $2^{k}$ divides $(m\pm j)/2$. Moreover, $q_\ell$ divides
  $(q_m\pm q_j)/2$ because $q_m$, $q_\ell$, and $q_j$ are all odd. Hence
  $2^k q_\ell=\ell$ divides $(m \pm j)/2$.

  Now, arguing as in \eqref{eq:mu} we have
  \begin{multline*}
    \sum_{j\in B} ( \omega_{2m}^{-j} + \omega_{2m}^{-(2m-j)}) = -\sum_{j\in
      B} ( \omega_m^{\frac{m-j}{2}} + \omega_m^{\frac{m+j}{2}})= -
    \sum_{\substack{h\in [m]\\ \gcd(m,h)=\ell}} \omega_m^{h} = -\mu(a),
  \end{multline*}
  and so $\left(\Ind_{\la d \ra}^{W_m} (\chi_{\zeta_d})\right) _{(\overline
    a^{\ell})} = -a^{\ell-1} (\ell-1)!\, 2^{\ell-1} \mu(a)$. This completes
  the proof of the lemma.
\end{proof}

%%%%%%%%%%%%%%%%%%%%%%%%%%%%%%%%%%%%%%%%%%%%%%%%%%%%%%%%%%%%%%%
%%%% Section 5 Main Result %%%%%%%%%%%
%%%%%%%%%%%%%%%%%%%%%%%%%%%%%%%%%%%%%%%%%%%%%%%%%%%%%%%%%%%%%%%
\section{Proof of \texorpdfstring{\cref{thm:main}}{cref}}\label{sec:mainthm}

Throughout this section, 
\[
\lambda= (\lambda_1, \dots, \lambda_a, \lambda_{a+1}, \dots, \lambda_{a+b})
\]
is a signed partition of $n$ and we use the notation and conventions
introduced in \cref{subsec:compandpart}.

%%%%%%%%%%%%%%%%%%%%%%%%%%%%%%%%%%%%%%%%%%%%%%%%%%
\subsection{}

To show that $Z_{W_n}(w_{\lambda})$ acts on $\elt$ on the right as scalars
we compute the action of the generators of $Z_{W_n}(w_{\lambda})$. Recall
from \cref{subsec:cent} that $Z_{W_n}(w_{\lambda})$ is generated by
\[
\{\, c_i, w_{0,\Lambda_i} \mid i\in [a] \,\} \amalg \{\, d_i \mid i\in [a+b]
\setminus [a] \,\} \amalg \{\, y_i \mid \lambda_i = \lambda_{i+1}, \
i\in[a+b-1] \,\},
\]
and from \cref{subsec:thm} that $\elt = \xlp f_1 \cdots f_{a+b}$, where 
\[
f_i =
\begin{cases}
  \epsilon_{\Lambda_i}^{+}\zeta_{c_i} & \text{if $i\in [a]$} \\
  \epsilon_{\Lambda_i}^{-} \tilde \zeta_{c_i} & \text{if $i \in
    [a+b]\setminus [a]$ and $\lambda_i$ is odd} \\
  \zeta_{d_i} & \text{if $i \in [a+b]\setminus [a]$ and $\lambda_i$ is
    even.}
\end{cases}
\]
In the computations below we repeatedly use the fact that $W_{\Lambda_i}$
and $W_{\Lambda_j}$ commute elementwise whenever $i\ne j$.

%%%%%%%%%%%%%%%%%%%%%%%%%%%%%%%%%%%%%%%%%%%%%%%%%%
\medskip We first show that $\elt \cdot c_i= \omega_{|c_i|} \elt$ for $i\in
[a]$: This equality is clear because $f_i= \epsilon_{\Lambda_i}^+
\zeta_{c_i}$ and so
\[
f_i\cdot c_i= \epsilon_{\Lambda_i}^+ \zeta_{c_i}\cdot c_i =
\omega_{\lambda_i} \epsilon_{\Lambda_i}^+ \zeta_{c_i} = \omega_{|c_i|} f_i.
\]
Thus
\[
\elt \cdot c_i = \xlp f_1\dotsm (f_ic_i ) \dotsm f_{a+b} = \omega_{|c_i|}
\xlp f_1\dotsm f_{a+b} = \omega_{|c_i|} \elt.
\]

%%%%%%%%%%%%%%%%%%%%%%%%%%%%%%%%%%%%%%%%%%%%%%%%%%
\medskip Next, we show that $\elt \cdot w_{0,\Lambda_i}= \elt$ for $i\in
[a]$: In this case $f_i= \epsilon_{\Lambda_i}^+ \zeta_{c_i} = ( 1/2)(\id+
w_{0, \Lambda_i}) \zeta_{c_i}$. Therefore,
\[
\epsilon_{\Lambda_i}^+ \zeta_{c_i} \cdot w_{0,\Lambda_i} =
(\epsilon_{\Lambda_i}^+ \cdot w_{0, \Lambda_i}) \cdot \zeta_{c_i} = (1/2) (
w_{0, \Lambda_i}+\id) \cdot \zeta_{c_i} = \epsilon_{\Lambda_i}^+
\zeta_{c_i},
\]
and the result follows.

%%%%%%%%%%%%%%%%%%%%%%%%%%%%%%%%%%%%%%%%%%%%%%%%%%
\medskip Now, we show that $\elt \cdot d_i= \omega_{|d_i|} \elt$ for $i\in
[a+b]\setminus [a]$: There are two cases depending on whether $\lambda_i$ is
odd or even. If $\lambda_i$ is odd, then $d_i= c_i w_{0, \Lambda_i}$ and
$f_i= \epsilon_{\Lambda_i}^- \tilde \zeta_{c_i} = (1/2) (\id- w_{0,
  \Lambda_i}) \tilde \zeta_{c_i}$. Therefore
\begin{multline*}
  f_i \cdot d_i = (1/2) (\id- w_{0, \Lambda_i}) \tilde \zeta_{c_i}\cdot
  c_iw_{0, \Lambda_i} = \big( (1/2)(\id- w_{0, \Lambda_i}) \cdot w_{0,
    \Lambda_i} \big) \big( \tilde \zeta_{c_i} \cdot c_i\big)\\ =
  -\omega_{|\lambda_i|}^{(|\lambda_i|+1)/2} \epsilon_{\Lambda_i}^- \tilde
  \zeta_{c_i} = \omega_{|d_i|} f_i.
\end{multline*}
On the other hand, if $\lambda_i$ is even, then $f_i= \zeta_{d_i}$ and the
result follows as in the computation of $\elt \cdot c_i$.

%%%%%%%%%%%%%%%%%%%%%%%%%%%%%%%%%%%%%%%%%%%%%%%%%%
\medskip Last, we show that $\elt \cdot y_i= \elt$ for $i\in [a+b-1]$ with
$\lambda_i=\lambda_{i+1}$: Recall that $y_i|_{\Lambda_l}$ is the identity
for $l\ne i, i+1$ and that the restriction of $y_i$ to ${\Lambda_i}$ defines
the unique order preserving bijection between $\Lambda_i$ and
$\Lambda_{i+1}$. In particular, $y_i$ is an involution in $S_n$ such that
$y_ic_i y_i=c_{i+1}$, $y_iw_{0, \Lambda_i} y_i=w_{0,\Lambda_{i+1}}$, and
$y_id_iy_i=d_{i+1}$. It is straightforward to check that $y_i f_i y_i =
f_{i+1}$, and so
\[
\elt \cdot y_i = \xlp f_1 \cdots f_{i} f_{i+1} y_i \cdots f_{a+b} = \xlp f_1
\cdots y_i f_{i+1} f_i \cdots f_{a+b} = \xlp y_i f_1 \cdots f_{a+b}.
\]

To complete the computation we need to show that $\xlp y_i= \xlp$. To see
this, recall that $\xlp = \sum_{w\in X_{|\lambda|}} w$.  Thus, it suffices
to show that $X_{|\lambda|} y_i= X_{|\lambda|}$. As in \cref{subsec:gridem},
if $w\in W_n$, then $w\in X_{|\lambda|}$ if and only if $w(j)>0$ for $j\in
[n]$ and $w|_{\Lambda_l}\colon \Lambda_l \to [n]$ is increasing for $l\in
[a+b]$. It is easy to see that if $w\in X_{|\lambda|}$, then $wy_i(j)>0$ for
$j\in [n]$ and $wy_i|_{\Lambda_l} \colon \Lambda_l \to [n]$ is increasing
for $l\in [a+b]$, so $wy_i\in X_{|\lambda|}$. Therefore $X_{|\lambda|} y_i=
X_{|\lambda|}$, and so $\xlp y_i= \xlp$.

%%%%%%%%%%%%%%%%%%%%%%%%%%%%%%%%%%%%%%%%%%%%%%%%%%
\subsection{} \label{subsec:end}

To complete the proof of \cref{thm:main}, it remains to show that there is
an isomorphism of right $\mathbb{C}W_n$-modules
\begin{equation*}
  \label{eq:eelt}
  E_{\lambda} \mathbb{C}W_n \cong  \Ind_{Z_{W_n}\left(w_\lambda
    \right)}^{W_n} \left( \CC\elt \right). 
\end{equation*}
Recall the idempotents $E_{\lambda}$ and the quasi-idempotents $e_p$ from
\cref{subsec:idemd}. Notice that $e_\lambda$ is one of the summands in the
definition of $E_\lambda$. The strategy is to show that
\begin{equation}
  \label{eq:eelt1}
  E_{\lambda}\mathbb{C}W_n = e_{\lambda} \mathbb{C}W_n \cong \elt
  \mathbb{C}W_n \cong \Ind_{Z_{W_n}\left(w_\lambda \right)}^{W_n} \left(
    \CC \elt \right).  
\end{equation}

\begin{lemma}\label{lem:El=el} 
  With the preceding notation,
  \[
  E_{\lambda}\mathbb{C}W_n = e_{\lambda} \mathbb{C}W_n.
  \]
\end{lemma}

\begin{proof}
  It follows from \cref{cor:ep2} that $E_{\lambda} e_\lambda = e_\lambda$
  and $e_\lambda E_\lambda=|\Stab (\lambda)|\, E_\lambda$, so
  \[
  e_\lambda \mathbb{C}W_n = E_{\lambda} e_\lambda \mathbb{C}W_n \subseteq
  E_{\lambda} \mathbb{C}W_n, \quad \text{and} \quad E_{\lambda}
  \mathbb{C}W_n = e_\lambda E_{\lambda} \mathbb{C}W_n \subseteq e_\lambda
  \mathbb{C}W_n.
  \]
\end{proof}

The next lemma is well-known and easy to prove. 

\begin{lemma}\label{lem:idemef=e}
  Suppose $e$ and $f$ are idempotents in a ring $A$.
  \begin{enumerate}
  \item If $ef=f$ and $fe=e$, then $eA = fA$.
  \item If $ef=e$ and $fe=f$, then $eA \cong fA$ as right ideals.
  \end{enumerate}
\end{lemma}

\begin{lemma}\label{lem:eltcongel}
  There is an isomorphism of right ideals
  \[
  e_{\lambda}\mathbb{C}W_n \cong\elt\mathbb{C}W_n.
  \]
\end{lemma}

\begin{proof}
  In this proof we use the theory of Lie idempotents (see
  \cite[\S8.4]{reutenauer:free}).

  Suppose first that $m$ is any positive integer. Then by \cite[Theorem
  8.16, Theorem 8.17]{reutenauer:free}, the Reutenauer idempotent $r_m$ and
  the Klyachko idempotent $\kappa_m$ are both Lie idempotents in the group
  algebra $\CC S_m$. Say $r_m=\sum_{x\in S_m} a_x x$ and
  $\kappa_m=\sum_{x\in S_m} \alpha_x x$. With the notation of
  \cref{lem:idemchar},
  \[
  \chi_{r_m}(w) = |Z_{S_m}(w)| \sum_{x\in \ccl(w)} a_x \quad
  \text{and}\quad\chi_\kappa(w) = |Z_{S_m}(w)| \sum_{x\in \ccl(w)} \alpha_x
  \]
  for $w\in W$. Garsia \cite[Proposition 5.1]{garsia:combinatorics} has
  shown that $\sum_{x\in \ccl(w)} a_x = \sum_{x\in \ccl(w)} \alpha_x$, which
  implies that $\chi_{r_m}= \chi_{\kappa_m}$, and hence that $r_m \CC S_m
  \cong \kappa_m \CC S_m$. In addition, Reutenauer \cite[Lemma 8.19]
  {reutenauer:free} has shown that if $c$ is the $m$-cycle $( 1 \; 2 \;
  \cdots \; m)$, $\omega$ is any primitive $m^{\text{th}}$ root of unity
  (not necessarily $\omega_m$), and $\zeta_c'= \frac 1m \sum_{j=1}^m
  \omega^{-j} c^j$, then $\zeta_c'\kappa = \kappa$, and $\kappa \zeta_c' =
  \zeta_c'$. Thus by \cref{lem:idemef=e} $\kappa_m \mathbb{C}S_m = \zeta_c'
  \mathbb{C}S_m$, which implies that
  \begin{equation}\label{eq:r=z}
    r_m\mathbb{C}S_m \cong \zeta_c' \mathbb{C}S_m.
  \end{equation}

  Suppose $i\in [a+b]$ and consider the right ideals
  $\epsilon_{\Lambda_i}^{\pm} r_{\Lambda_i} \CC W_{\Lambda_i}$ and $f_i \CC
  W_{\Lambda_i}$ in $\CC W_{\Lambda_i}$. If $i\in [a]$, then $f_i=
  \epsilon_{\Lambda_i}^{+} \zeta_{c_i}$, and it follows from \cref{eq:r=z}
  that $\epsilon_{\Lambda_i}^{+} r_{\Lambda_i} \CC W_{\Lambda_i} \cong
  \epsilon_{\Lambda_i} ^{+} \zeta_{c_i} \CC W_{\Lambda_i} = f_i \CC
  W_{\Lambda_i}$.  If $i\in [a+b]\setminus [a]$ and $\lambda_i$ is odd, then
  $f_i= \epsilon_{\Lambda_i}^{-} \tilde \zeta_{c_i}$, and it again follows
  from \cref{eq:r=z} that $\epsilon_{\Lambda_i}^{-} r_{\Lambda_i} \CC
  W_{\Lambda_i} \cong \epsilon_{\Lambda_i}^{-} \tilde \zeta_{c_i} \CC
  W_{\Lambda_i} = f_i \CC W_{\Lambda_i}$. Finally, if $i\in [a+b]\setminus
  [a]$ and $\lambda_i$ is even, then $f_i= \zeta_{d_i}$, and it follows from
  \cref{eq:r=z} and \cref{prop:evenind} that $\epsilon_{\Lambda_i}^{-}
  r_{\Lambda_i} \CC W_{\Lambda_i} \cong \epsilon_{\Lambda_i}^{-} \zeta_{c_i}
  \CC W_{\Lambda_i} \cong \zeta_{d_i} \CC W_{\Lambda_i}= f_i \CC
  W_{\Lambda_i}$. Thus
  \begin{equation}\label{eq:r=f}
    \epsilon_{\Lambda_i}^{\pm} r_{\Lambda_i} \CC W_{\Lambda_i}
    \cong f_i \CC W_{\Lambda_i}  
  \end{equation}
  for all $i\in [a+b]$.

  To complete the proof we use \eqref{eq:r=f} to compute
  \begin{align*}
    e_\lambda \mathbb{C}W_n 
    &= x_{|\lambda|}\, \epsilon_{\Lambda_1}^{+} r_{\Lambda_1} \dotsm
      \epsilon_{\Lambda_{a+b}}^{-} r_{\Lambda_{a+b}}
      \mathbb{C}W_n \\ 
    &= \xlp \big(\epsilon_{\Lambda_1}^{+} r_{\Lambda_1} 
      \mathbb{C} W_{\Lambda_1}\big) \cdots
      \big( \epsilon_{\Lambda_{a+b}}^{-} r_{\Lambda_{a+b}}
      \mathbb{C} W_{\Lambda_{a+b}}\big) \cdot
      \mathbb{C}W_n\\  
    &\cong \xlp \big(f_1 \mathbb{C} W_{\Lambda_1}\big) \cdots
      \big(f_{a+b} \mathbb{C} W_{\Lambda_{a+b}}\big) \cdot
      \mathbb{C}W_n  \\
    &= \elt \mathbb{C}W_n.
  \end{align*}
\end{proof}

The last isomorphism in \cref{subsec:end}\cref{eq:eelt1} follows from the
next lemma.

\begin{lemma}
  The multiplication map $\mathbb{C} \elt \otimes_{Z_{W_n}(w_\lambda)}
  \mathbb{C}W_n \rightarrow \elt \mathbb{C}W_n$ is an isomorphism of right
  $\CC W_n$-modules.
\end{lemma}

\begin{proof}
  The mapping is obviously $\CC W_n$-linear and surjective, so
  \begin{equation}\label{eq:1a}
    \dim \left( \elt \mathbb{C}W_n \right) \leq
    \dim \left(\mathbb{C} \elt \otimes_{Z_{W_n}(w_\lambda)}
      \mathbb{C}W_n\right) = |W_n|/ |Z_{W_n}(w_\lambda)|.
  \end{equation}

  Now, using the decomposition \cref{subsec:eqdecomp}\cref{eq:decomp} and
  the isomorphism $E_\mu \CC W_n \cong \widetilde e_\mu \CC W_n$ from
  \cref{lem:El=el} and \cref{lem:eltcongel}, we have
  \begin{multline*}
    \dim \mathbb{C}W_n = \sum_{\mu\in \CalS\CalP(n)} \dim E_\mu
    \mathbb{C}W_n = \sum_{\mu\in \CalS\CalP(n)} \dim \widetilde e_\mu
    \mathbb{C}W_n \\ \leq \sum_{\mu\in \CalS\CalP(n)} |W_n|/
    |Z_{W_n}(w_\mu)| = |W_n|,
  \end{multline*}
  and so it follows from \eqref{eq:1a} that $\dim \widetilde e_\mu
  \mathbb{C}W_n = |W_n|/ |Z_{W_n}(w_\mu)|$ for $\mu\in \CalS
  \CalP(n)$. Therefore $\dim \elt \mathbb{C}W_n = \dim \left( \mathbb{C}
    \elt \otimes_{\mathbb{C} Z_{W_n}(w_\lambda)} \mathbb{C}W_n \right)$, and
  it follows that the multiplication map in the statement of the lemma is an
  isomorphism as claimed.
\end{proof}

%%%%%%%%%%%%%%%%%%%%%%%%%%%%%%%%%%%%%%%%%%%%%%%%%%%%%%%%%%%%%%%
%%%% Section 6
%%%%%%%%%%%%%%%%%%%%%%%%%%%%%%%%%%%%%%%%%%%%%%%%%%%%%%%%%%%%%%%
\section{Computing \texorpdfstring{$\theta_n(E_\lambda)$}
  {tEl}} \label{sec:theta}

Recall from \cref{subsec:bonn} that $\cfC(W_n)$ denotes the algebra of
$\CC$-valued class functions on $W_n$, and that for $\mu \in \CalS\CalP(n)$,
$u_\mu$ is the characteristic function for $\ccl(w_{\mu})$. Recall also the
surjective algebra homomorphism $\theta_n\colon \Sigma(W_n) \rightarrow
\cfC(W_n)$ with kernel equal to the Jacobson radical of $\Sigma(W_n)$. In
this section we prove \cref{thm:bonnperm}:
\begin{quotation}
  \emph{Suppose $\lambda\in \CalS\CalP(n)$.  Then $\theta_n(E_\lambda)=
    u_{\overleftarrow{\lambda'}}$, where $\lambda'$ is the signed
    composition of $n$ defined by
    \[
    \lambda_i' = \begin{cases} \lambda_i &\text{if $\lambda_i$ is odd} \\
      \overline{\lambda_i} &\text{if $\lambda_i$ is even.}
    \end{cases}
    \]
  }
\end{quotation}

The proof requires several preliminary results. To begin, as observed in
\cref{subsec:bonn}, $\{\, \theta_n(E_\mu) \mid \mu \in \CalS\CalP(n)\, \} =
\{\, u_\mu\mid \mu \in \CalS\CalP(n)\, \}$ is the basis of $\cfC(W_n)$
formed by the primitive idempotents because $\theta_n$ identifies
$\cfC(W_n)$ with the semisimple quotient of $\Sigma(W_n)$ by its Jacobson
radical and $\cfC(W_n)$ is a commutative algebra. Thus,
$\theta_n(E_\lambda)= u_{\mu}$ for some $\mu$. The first reduction is to
replace $E_\lambda$ by $e_\lambda$.

\begin{lemma}\label{lem:61}
  With the preceding notation, $\theta_n(E_\lambda) = |\Stab
  (\lambda)|\inverse \theta_n(e_\lambda)$.
\end{lemma}

\begin{proof}
  It follows from the definition of $E_\lambda$ in \cref{subsec:eqdecomp},
  and \cref{cor:ep2}, that $e_\lambda =E_\lambda e_\lambda$ and $e_\lambda
  E_\lambda = |\Stab (\lambda)|\, E_\lambda$. Therefore,
  \[
  \theta_n(e_\lambda) = \theta_n(E_\lambda e_\lambda) = \theta_n(E_\lambda)
  \theta_n(e_\lambda) = \theta_n(e_\lambda) \theta_n(E_\lambda) =
  \theta_n(e_\lambda E_\lambda) = |\Stab (\lambda)|\, \theta_n(E_\lambda).
  \]
\end{proof}

%%%%%%%%%%%%%%%%%%%%%%%%%%%%%%%%%%%%%%%%%%%%%%%%%%%%%%%%%%%%%%%
\subsection{}
If $\chi\in \cfC(W_n)$, then $\chi = \sum_{\mu\in \CalS\CalP(n)}
|Z_{W_n}(w_\mu)| \la \chi, u_{\mu} \ra_{W_n} \cdot u_\mu$, where
$\la\,\cdot\, , \,\cdot\, \ra_{W_n}$ is the usual inner product on
$\cfC(W_n)$, so to prove the theorem it is enough to compute $\la
\theta_n(e_\lambda), u_{\mu} \ra_{W_n}$ for $\mu\in \CalS\CalP(n)$.  More
generally, in \cref{prop:tepmu} we give an explicit formula for $\la
\theta_n(e_p), u_{\mu} \ra_{W_n}$. The first step is to give a formula for
$\theta_m( \epsilon_m^\pm r_m)$.

%%%%%%%%%%%%%%%%%%%%%%%%%%%%%%%%%%%%%%%%%%%%%%%%%%%%%%%%%%%%%%%
\begin{prop}\label{prop:rm} 
  Let $m$ be a positive integer. Then
  \begin{gather}
    \label{eq:gata}
    \theta_m(r_m)= u_{(m)} + u_{(\overline m)}, \\
    \intertext{and} \theta_m( \epsilon_m^+ r_m) = \begin{cases} u_{(m)}
      &\text{if $m$ is odd} \\
      u_{(\overline m)} &\text{if $m$ is even} \end{cases}
    \quad\text{and}\quad \theta_m( \epsilon_m^- r_m) = \begin{cases} u_{(m)}
      &\text{if $m$ is even} \\
      u_{(\overline m)} &\text{is $m$ is odd.} \end{cases} \label{eq:gatb}
  \end{gather}
\end{prop}

\begin{proof} 
  Recall that $r_m$ lies in the descent algebra of $S_m$ and that there is a
  surjective algebra homomorphism $\theta_{S_m}$ from the descent algebra of
  $S_m$ to the algebra of class functions $\cfC(S_m)$. It follows from the
  results in \cite[\S3]{garsiareutenauer:decomposition} that
  $\theta_{S_m}(r_m) = u_{(m)}^{S_m}$, where $u_{(m)}^{S_m}$ is the
  characteristic function of the conjugacy class of $m$-cycles in $S_m$.

  By \cite[(3.4)] {bonnafehohlweg:generalized}, $\theta_m(r_m)$ is the lift
  to $\cfC(W_m)$ of $u_{(m)}^{S_m}$, so if $\pi\colon W_m\to S_m$ is the
  projection with kernel equal to $T= \la t_1, \dots, t_m \ra$, then
  $\theta_m(r_m) = u_{(m)}^{S_m} \circ \pi$. For $w\in S_m$ and $t\in T$,
  $\pi(wt)=w$, so $u_{(m)}^{S_m} \circ \pi (wt) =0$ unless $w$ is an
  $m$-cycle, in which case $u_{(m)}^{S_m} \circ \pi (wt) =1$. On the other
  hand, if $w$ is an $m$-cycle and $t\in T$, then $wt$ has signed cycle type
  $(m)$ or $(\overline m)$. It follows that $\theta_m(r_m)= u_{(m)} +
  u_{(\overline m)}$.

  Now, let $\varepsilon_m$ denote the sign character of $W_m$. By
  \cite[Example 3.5]{bonnafehohlweg:generalized}, $\theta_m (w_{0,m}) =
  \varepsilon_m$.  Then, using \eqref{eq:gata} and the fact that $\theta_m$
  is an algebra homomorphism, we have that
  \[
  \theta_m(\epsilon_m^\pm r_m) = (1/2)(\id \pm \varepsilon_m) ( u_{(m)} +
  u_{(\overline m)}) = (1/2)( u_{(m)} + u_{(\overline m)} \pm \varepsilon_m
  u_{(m)} \pm \varepsilon_m u_{(\overline m)}).
  \]
  
  One checks that
  \[
  \varepsilon_m u_{(m)} =
  \begin{cases}
  u_{(m)} & \text{if $m$ is odd} \\
  -u_{(m)} & \text{if $m$ is even}
  \end{cases}
  \quad\text{and} \quad \varepsilon_m u_{(\overline m)} =
  \begin{cases}
    u_{(\overline m)} & \text{if $m$ is even} \\
    -u_{(\overline m)} & \text{if $m$ is odd.}
  \end{cases}
  \]
  Hence, if $m$ is odd, then
  \begin{equation}
    \label{eq:evep}
    \theta_m (\epsilon_m^\pm r_m) = (1/2) ( u_{(m)} + u_{(\overline m)} \pm
    u_{(m)} \mp u_{(\overline m)}) =
    \begin{cases}
      u_{(m)}&\text{for $\epsilon_m^+$} \\ u_{(\overline m)} &\text{for
        $\epsilon_m^-$,}
    \end{cases}
  \end{equation}
  and if $m$ is even, then
  \begin{equation}
    \label{eq:odep}
    \theta_m (\epsilon_m^\pm r_m) = (1/2)( u_{(m)} + u_{(\overline m)} \mp
    u_{(m)} \pm u_{(\overline m)}) =
    \begin{cases}
      u_{(m)}&\text{for $\epsilon_m^-$} \\ u_{(\overline m)} &\text{for
        $\epsilon_m^+$.}
    \end{cases}
  \end{equation}
  The formulas in \eqref{eq:gatb} follow from \eqref{eq:evep} and
  \eqref{eq:odep}.
\end{proof}

%%%%%%%%%%%%%%%%%%%%%%%%%%%%%%%%%%%%%%%%%%%%%%%%%%%%%%%%%%%%%%%
\subsection{}
In this subsection and the next, $p=(p_1, \dots, p_k)$ denotes a fixed
signed composition of $n$.  Our goal is to compute $\la \theta_n(e_p),
u_{\mu} \ra_{W_n}$ for $\mu\in \CalS\CalP(n)$. Recall that $e_p= x_{|p|}\,
\epsilon_{P_1}^{\xi_1} r_{P_1} \dotsm \epsilon_{P_{k}}^{\xi_{k}} r_{P_{k}}$,
where $\xi_i$ is the sign of $p_i$,
\[
\epsilon_{P_1}^{\xi_1} r_{P_1} \dotsm \epsilon_{P_{k}}^{\xi_{k}} r_{P_{k}}
\in \CC W_{|p|}, \qquad W_{|p|}=W_{P_1} \dotsm W_{P_k} \cong W_{|p_1|}
\times \dotsm \times W_{|p_k|},
\]
and $W_{|p_i|} \cong W_{P_i}\subseteq W_n$.

For $i\in [k]$ define an isomorphism $f_i\colon W_{|p_i|} \xrightarrow{\
  \cong\ } W_{P_i}$ by
\[
( f_i ( w)) ( l) =
\begin{cases}
  \widehat p_{i-1} + w ( l - \widehat p_{i-1}) & \text{if $l \in P_i,\ w( l
    - \widehat p_{i-1})>0$} \\
  \overline{\widehat p_{i-1}} + w( l - \widehat
  p_{i-1}) & \text{if $l\in P_i,\ w(l - \widehat p_{i-1}) <0$} \\
  l&\text{otherwise}
\end{cases}
\]
for $w$ in $W_{|p_i|}$ and $l\in [p_i]$. Then $f_i (w)$ is the identity on
$[n]\setminus P_i$, and the restriction of $f_i (w)$ to $P_i$ is the
translation of $w$ from a map $[p_i]\to \pm[p_i]$ to a map $P_i\to \pm P_i$.

The embeddings $f_1$, \dots, $f_k$ define a group isomorphism
\[
f=f_1\times \dots \times f_k\colon W_{|p_1|} \times \cdots \times W_{|p_k|}
\xrightarrow{\ \cong\ } W_{|p|},
\]
and an algebra isomorphism (also denoted by $f$)
\begin{equation}
  \label{eq:gai}
  f\colon \CC W_{|p_1|} \otimes \cdots \otimes \CC W_{|p_k|}
  \xrightarrow{\ \cong\ } \CC W_{|p|}.  
\end{equation}
Because $W_{|p|}$ is the internal product $W_{P_1} \dotsm W_{P_k}$, the
isomorphism in \eqref{eq:gai} restricts to an isomorphism (still denoted by
$f$)
\[
f\colon \cfC(W_{|p_1|}) \otimes \dotsm \otimes \cfC(W_{|p_{k}|})
\xrightarrow{\ \cong\ } \cfC(W_{|p|}).
\]
Notice that 
\begin{equation*}
  \label{eq:fcfbox}
  f(\eta_1\otimes \dotsm \otimes \eta_k) = \eta_1f_1\inverse
  \boxtimes \dotsm \boxtimes \eta_k f_k\inverse\in \cfC(W_{|p|}) , 
\end{equation*}
where $(\phi_1\boxtimes \dotsm \boxtimes \phi_k) (v_1\dotsm v_k)=
\phi_1(v_1) \dotsm \phi_k(v_k)$ for $i\in [k]$, $\eta_i\in \cfC(W_{|p_i|})$,
$\phi_i\in \cfC(W_{P_i})$, and $v_i\in W_{P_i}$.

%%%%%%%%%%%%%%%%%%%%%%%%%%%%%%%%%%%%%%%%%%%%%%%%%%%%%%%%%%%%%%%
\subsection{}
For $q\in \CalS\CalC(n)$ with $W_q\subseteq W_{|p|}$, set $X_q^{|p|} = X_q
\cap W_{|p|}$, and let $x_q^{|p|} = \sum_{w\in X_q^{|p|}} w$. Bonnaf\'e and
Hohlweg \cite[Section 3.1]{bonnafehohlweg:generalized} define
\[
\Sigma'(W_{|p|}) =\spn \{ x_q^{|p|} \mid q\in \CalS\CalC(n), \ W_q \subseteq
W_{|p|} \,\}.
\]
They show that $\{ x_q^{|p|} \mid q\in \CalS\CalC(n), \ W_q \subseteq
W_{|p|} \,\}$ is a basis of $\Sigma'(W_{|p|})$, that $\Sigma'(W_{|p|})$ is a
subalgebra of $\mathbb{C}W_n$, and that there is an algebra homomorphism
\[
\theta_{|p|}\colon \Sigma'(W_{|p|}) \to \cfC(W_{|p|})
\]
with the same properties as $\theta_n$. (More generally, Bonnaf\'e and
Hohlweg consider subalgebras $\Sigma'(W_p)$ and homomorphisms $\theta_{p}$,
and $\Sigma(W_n)$ and $\theta_n$ are defined as the special case when
$p=(n)$.)

%%%%%%%%%%%%%%%%%%%%%%%%%%%%%%%%%%%%%%%%%%%%%%%%%%%%%%%%%%%%%%%
\begin{lemma}\label{lem:bh}
  With the preceding notation,
  \[
  f\big( \Sigma(W_{|p_1|}) \otimes \dotsm \otimes \Sigma(W_{|p_k|}) \big) =
  \Sigma'(W_{|p|}),
  \]
  and the diagram
  \begin{equation}
    \label{eq:cd2}
    \vcenter{\vbox{\xymatrix{
      \Sigma(W_{|p_1|}) \otimes \cdots \otimes \Sigma(W_{|p_{k}|})
      \ar[r]^-{f}_-{\cong} \ar[d]^-{\theta_{|p_1|}\otimes \cdots \otimes 
        \theta_{|p_k|}} & \Sigma'(W_{\abs{p}}) \ar[d]^-{\theta_{\abs{p}}} \\
      \cfC(W_{\abs{p_1}}) \otimes \dotsm \otimes \cfC(W_{\abs{p_{k}}})
      \ar[r]^-{f}_-{\cong} & \cfC(W_{\abs{p}})} }}
  \end{equation}
  commutes.
\end{lemma}

\begin{proof}
  Suppose that $q^i$ is a signed composition of $|p_i|$ for $i\in [k]$ and
  that $q$ is the concatenation of $q^1$, \dots, $q^k$.  Then $q\in
  \CalS\CalC(n)$ and $X_{q^i} \subseteq W_{|p_i|}$. Straightforward
  computations using \cite[Remark 2.1]{bonnafehohlweg:generalized} and the
  definitions show that
  \begin{equation}
    \label{eq:xqpi}
    f_i ( X_{q^i}) = X_{\underline q^i}^{|p|},\quad\text{where} \quad
    \underline q^i = ( |p_1|, \ldots, |p_{i-1}|, q^i, |p_{i+1}|, \ldots,
    |p_k|) \in \CalS\CalC(n),
  \end{equation}
  and that
  \begin{equation}
    \label{eq:xqpt}
    X_{\underline q^1}^{|p|} \dotsm X_{\underline q^k}^{|p|} = X_q^{|p|} . 
  \end{equation}
  It follows from \eqref{eq:xqpi} and \eqref{eq:xqpt} that
  \[
  f(x_{q^1} \otimes \dotsm \otimes x_{q^k}) = x_{\underline q^1} ^{|p|}
  \dotsm x_{\underline q^k} ^{|p|} = x_q ^{|p|}.
  \]

  One checks that the rule $(q^1,\dots, q^k) \mapsto q$ defines a bijection
  \begin{equation*}
    \label{eq:xqpc}
    \CalS\CalC(|p_1|) \times \dotsm \times \CalS\CalC (|p_k|)
    \leftrightarrow \{\, q\in \CalS\CalC(n)\mid W_q\subseteq W_{|p|}\,\},
  \end{equation*}
  and so $f$ maps the basis
  \[
  \{\, x_{q^1} \otimes \dotsm \otimes x_{q^k} \mid \forall\, i\in [k],\
  q^i\in \CalS\CalC(|p_i|) \,\} \quad\text{of}\quad \Sigma(W_{|p_1|})
  \otimes \dotsm \otimes \Sigma(W_{|p_1|})
  \]
  to the basis
  \[
  \{ x_q^{|p|} \mid q\in \CalS\CalC(n), \ W_q \subseteq W_{|p|}
  \,\}\quad\text{of}\quad \Sigma'(W_{|p|}).
  \]
  Therefore 
  \[
  f(\Sigma(W_{|p_1|}) \otimes \dotsm \otimes \Sigma(W_{|p_k|}) ) =
  \Sigma'(W_{|p|})
  \]
  as claimed.

  Finally, 
  \begin{align*}
    \theta_{|p|} (f(x_{q^1} \otimes \dotsm \otimes x_{q^k}))
    &= \theta_{|p|}(x_q^{|p|}) \\
    &= \Ind_{W_q} ^{W_{|p|}}( 1_{W_q}) \\
    &= \Ind_{f_1(W_{q^1})} ^{W_{P_1}}( 1_{f_1(W_{q^1})}) \boxtimes \dotsm
      \boxtimes \Ind_{f_k(W_{q^k})} ^{W_{P_k}}( 1_{f_k(W_{q^k})}) \\
    &= \Ind_{W_{q^1}} ^{W_{|p_1|}}( 1_{W_{q^1}}) f_1\inverse \boxtimes \dotsm
      \boxtimes \Ind_{W_{q^k}} ^{W_{|p_k|}}( 1_{W_{q^k}}) f_k\inverse \\
    &= f\big( \Ind_{W_{q^1}} ^{W_{|p_1|}}( 1_{W_{q^1}}) \otimes \dotsm
      \otimes \Ind_{W_{q^k}} ^{W_{|p_k|}}( 1_{W_{q^k}})\big) \\
    &= f\big( \theta_{|p_1|} ( x_{q^1}) \otimes \dotsm \otimes
      \theta_{|p_k|}( x_{q^k}) \big) \\
    &= f\big( (\theta_{|p_1|} \otimes \dotsm \otimes \theta_{|p_k|} )(
      x_{q^1} \otimes \dotsm \otimes x_{q^k}) \big),
  \end{align*}
  and it follows that \eqref{eq:cd2} commutes.
\end{proof}

We can now give a formula for $\la \theta_n(e_p), u_{\mu} \ra_{W_n}$.

%%%%%%%%%%%%%%%%%%%%%%%%%%%%%%%%%%%%%%%%%%%%%%%%%%%%%%%%%%%%%%%
\begin{prop}\label{prop:tepmu}
  Suppose $p=(p_1, \dots, p_k)$ is a signed composition of $n$ and $\mu$ is
  a signed partition of $n$. Then
  \[
  \la \theta_n(e_p), u_{\mu} \ra_{W_n} =
  \begin{cases}
    2^{-k}|p_1\dotsm p_k|\inverse &\text{if $\mu= \overleftarrow{p'}$} \\
    0 &\text{otherwise.}
  \end{cases}
  \]
\end{prop}

\begin{proof}
  Consider the diagram
  \begin{equation}\label{eq:commdiag}
    \vcenter{\vbox{\xymatrix{
      \Sigma(W_{\abs{p_1}}) \otimes \cdots \otimes \Sigma(W_{\abs{p_{k}}})
      \ar[rr]^-{f}_-{\cong} \ar[d]^{\theta_{|p_1|}\otimes \cdots \otimes
      \theta_{|p_k|}} && \Sigma'(W_{\abs{p}}) \ar[rr]^{x_{\abs{p}}*}
      \ar[d]^{\theta_{\abs{p}}} && \Sigma(W_n) \ar[d]^{\theta_n} \\
      \cfC(W_{\abs{p_1}}) \otimes \dotsm \otimes \cfC(W_{\abs{p_{k}}})
      \ar[rr]^-{f}_-{\cong} && \cfC(W_{\abs{p}})
      \ar[rr]^{\Ind_{W_{\abs{p}}}^{W_n}}&& \cfC(W_n) ,  }}}
  \end{equation}
  where $x_{|p|}*$ denotes left multiplication by $x_{|p|}$. It was shown in
  \cref{lem:bh} that the left square commutes and it is shown in
  \cite[Section 3.2]{bonnafehohlweg:generalized} that the right square
  commutes, so \eqref{eq:commdiag} is a commutative diagram.

  Using the commutativity of the right square we have
  \begin{equation}
    \label{eq:rhs}
    \theta_n(e_p) = \theta_n( x_{|p|} \epsilon_{P_1} ^{\xi_1} r_{P_1} \dotsm
    \epsilon_{P_k} ^{\xi_k} r_{P_k}) = \Ind_{W_{|p|}}^{W_n}\left(
      \theta_{|p|} ( \epsilon_{P_1} ^{\xi_1} r_{P_1} \dotsm \epsilon_{P_k}
      ^{\xi_k} r_{P_k}) \right),  
  \end{equation}
  and using the commutativity of the left square we have
  \begin{equation}
    \label{eq:lhs}
    \begin{aligned}
      \theta_{|p|} ( \epsilon_{P_1} ^{\xi_1} r_{P_1} \dotsm \epsilon_{P_k}
      ^{\xi_k} r_{P_k}) &=\theta_{|p|} \big( f( \epsilon_{|p_1|} ^{\xi_1}
      r_{|p_1|} \otimes \dotsm
      \otimes \epsilon_{|p_k|} ^{\xi_k} r_{|p_k|}) \big) \\
      &= f \big( (\theta_{|p_1|}\otimes \cdots \otimes \theta_{|p_k|}) (
      \epsilon_{|p_1|} ^{\xi_1} r_{|p_1|} \otimes \dotsm
      \otimes \epsilon_{|p_k|} ^{\xi_k} r_{|p_k|}) \big) \\
      &= \theta_{|p_1|} ( \epsilon_{|p_1|} ^{\xi_1} r_{|p_1|}) f_1\inverse
      \boxtimes \dotsm \boxtimes \theta_{|p_k|}(\epsilon_{|p_k|} ^{\xi_k}
      r_{|p_k|}) f_k \inverse.
    \end{aligned}
  \end{equation}

  Using \eqref{eq:rhs}, Frobenius reciprocity, and \eqref{eq:lhs} gives
  \begin{align*}
    \la \theta_n(e_p), u_\mu \ra_{W_n}
    &= \la \Ind_{W_{|p|}}^{W_n} \big(
      \theta_{|p|} ( \epsilon_{P_1} ^{\xi_1} r_{P_1} \dotsm \epsilon_{P_k}
      ^{\xi_k} r_{P_k}) \big),  u_\mu \ra_{W_n} \\
    & = \la \theta_{|p|} ( \epsilon_{P_1} ^{\xi_1} r_{P_1} \dotsm
      \epsilon_{P_k} ^{\xi_k} r_{P_k}) , u_\mu|_{W_{|p|}} \ra_{W_{|p|}} \\
    & = \la \theta_{|p_1|} ( \epsilon_{|p_1|} ^{\xi_1} r_{|p_1|}) f_1\inverse
      \boxtimes \dotsm \boxtimes \theta_{|p_k|}(\epsilon_{|p_k|} ^{\xi_k}
      r_{|p_k|}) f_k \inverse, u_\mu|_{W_{|p|}} \ra_{W_{|p|}} .
  \end{align*}
  One checks that $u_\mu|_{W_{|p|}} \ne 0$ if and only if for $i\in [k]$
  there are signed partitions $\mu^i$ of $|p_i|$ such that if $q\in
  \CalS\CalC(n)$ is the concatenation of $\mu^1$, \dots, $\mu^k$, then
  $\mu=\overleftarrow q$.  Suppose that this is the case. Then
  $u_\mu|_{W_{|p|}} = u_{\mu^1}f_1\inverse \boxtimes \dotsm \boxtimes
  u_{\mu^k}f_k\inverse$. Thus
  \begin{multline*}
    \la \theta_{|p_1|} ( \epsilon_{|p_1|} ^{\xi_1} r_{|p_1|}) f_1\inverse
    \boxtimes \dotsm \boxtimes \theta_{|p_k|}(\epsilon_{|p_k|} ^{\xi_k}
    r_{|p_k|}) f_k \inverse, u_\mu|_{W_{|p|}} \ra_{W_{|p|}} \\
    = \la \theta_{|p_1|} ( \epsilon_{|p_1|} ^{\xi_1} r_{|p_1|}) f_1\inverse
    \boxtimes \dotsm \boxtimes \theta_{|p_k|}(\epsilon_{|p_k|} ^{\xi_k}
    r_{|p_k|}) f_k \inverse, u_{\mu^1}f_1\inverse \boxtimes \dotsm \boxtimes
    u_{\mu^k}f_k\inverse \ra_{W_{|p|}} \\
    = \la \theta_{|p_1|} ( \epsilon_{|p_1|} ^{\xi_1} r_{|p_1|}) f_1\inverse,
    u_{\mu^1}f_1\inverse \ra_{W_{P_1}} \dotsm \la
    \theta_{|p_k|}(\epsilon_{|p_k|} ^{\xi_k} r_{|p_k|}) f_k \inverse,
    u_{\mu^k}f_k\inverse \ra_{W_{P_k}} \\
    = \la \theta_{|p_1|} ( \epsilon_{|p_1|} ^{\xi_1} r_{|p_1|}) , u_{\mu^1}
    \ra_{W_{|p_1|}} \dotsm \la \theta_{|p_k|} ( \epsilon_{|p_k|} ^{\xi_k}
    r_{|p_k|}) , u_{\mu^k} \ra_{W_{|p_k|}},
  \end{multline*}
  and so by \cref{prop:rm},
  \begin{align*}
    \la \theta_n(e_p), u_\mu \ra_{W_n}
    & = \la \theta_{|p_1|} ( \epsilon_{|p_1|} ^{\xi_1} r_{|p_1|}) ,
      u_{\mu^1} \ra_{W_{|p_1|}} \dotsm \la \theta_{|p_k|} ( \epsilon_{|p_k|}
      ^{\xi_k} r_{|p_k|}) , u_{\mu^k} \ra_{W_{|p_k|}} \\
    &= \prod_{\substack{p_i>0\\ p_i \text{ odd}}} \la u_{(p_i)}, u_{\mu^i}
    \ra_{W_{|p_i|}}\cdot \prod_{\substack{p_i>0\\ p_i \text{ even}}} \la
    u_{(\overline{p_i})}, u_{\mu^i}  \ra_{W_{|p_i|}} \\
    &\qquad \cdot 
      \prod_{\substack{p_i<0\\ p_i \text{ even}}} \la u_{(\overline{p_i})},
    u_{\mu^i} \ra_{W_{|p_i|}}\cdot \prod_{\substack{p_i<0\\ p_i \text{
    odd}}} \la u_{(p_i)}, u_{\mu^i} \ra_{W_{|p_i|}} \\
    &= \prod_{p_i \text{ odd}} \la u_{(p_i)}, u_{\mu^i} \ra_{W_{|p_i|}}\cdot
      \prod_{p_i \text{ even}} \la u_{(\overline{p_i})}, u_{\mu^i} 
      \ra_{W_{|p_i|}} \\ 
    &= \begin{cases} 2^{-k}|p_1\dotsm p_k|\inverse & \text{if $\mu=
        \overleftarrow {p'}$} \\
      0&\text{otherwise.} \end{cases} 
  \end{align*}
\end{proof}

%%%%%%%%%%%%%%%%%%%%%%%%%%%%%%%%%%%%%%%%%%%%%%%%%%%%%%%%%%%%%%%
\subsection{}
Now suppose $\lambda=(\lambda_1, \dots, \lambda_a, \lambda_{a+1}, \dots,
\lambda_{a+b})$ is a signed partition of $n$. Then $|\Stab (\lambda)| =
|\Stab (\overleftarrow{\lambda'})|$ and with the notation in
\cref{subsec:cent}, the subgroup of $Z_{W_n}(w_\lambda)$ generated by $\{\,
y_i\mid \lambda_i= \lambda_{i+1}\,\}$ is isomorphic to $\Stab
(\lambda)$. Thus
\[
|Z_{W_n}(w_{\overleftarrow{\lambda'}})| = |\Stab (\lambda)|\, 2^{a+b}\,
|\lambda_1 \dotsm \lambda_{a+b}|
\]
and so by \cref{lem:61} and \cref{prop:tepmu},
\begin{multline*}
  \theta(E_\lambda) = |\Stab (\lambda)|\inverse \theta_n(e_\lambda) = |\Stab
  (\lambda)|\inverse \sum_{\mu\in \CalS\CalP(n)} |Z_{W_n}(w_\mu)|\, \la
  \theta_n(e_\lambda), u_{\mu} \ra_{W_n} \cdot u_\mu\\
  = |\Stab (\lambda)|\inverse |Z_{W_n}(w_{\overleftarrow{\lambda'}})|\,
  2^{-a-b}\,|\lambda_1\dotsm \lambda_{a+b}|\inverse\cdot
  u_{\overleftarrow{\lambda'}} = u_{\overleftarrow{\lambda'}},
\end{multline*}
as claimed.

\medskip

\noindent {\bf Acknowledgments:} The authors thank Nantel Bergeron, G\"otz
Pfeiffer, and Monica Vazirani for helpful discussions.

%%%%%%%%%%%%%%%%%%%%%%%%%%%%%%%%%%%%%%%%%%%%%%%%%%%%%%%%%%%%%%%%%%%%%%
%%%%%%%%%%%%% bibliography
%%%%%%%%%%%%%%%%%%%%%%%%%%%%%%%%%%%%%%%%%%%%%%%%%%%%%%%%%%%%%%%%%%%%%%

\bibliographystyle{plain}
%%\bibliography{matts} 

\end{document}